\documentclass[12pt,centertags,oneside, reqno]{amsart}
\usepackage{ifpdf}

\usepackage{amssymb}
\usepackage{braket}
\usepackage[all]{xy}
\usepackage{cancel}

\usepackage{amscd,amsxtra,calc}
\usepackage{cmmib57}
\usepackage{url}
\usepackage{ulem}
\usepackage{xcolor}
\usepackage[a4paper,width=16.2cm,top=3cm,bottom=3cm]{geometry}

\numberwithin{equation}{section}

\theoremstyle{plain}
\newtheorem{thm}{Theorem}[section]
\newtheorem{theorem}[thm]{Theorem}

\newtheorem{lemma}[thm]{Lemma}
\newtheorem{corollary}[thm]{Corollary}
\newtheorem{proposition}[thm]{Proposition}
\theoremstyle{definition}
\newtheorem{remark}[thm]{Remark}

\newtheorem{definition}[thm]{Definition}
\newtheorem{claim}[thm]{Claim}

\newtheorem{question}[thm]{Question}

\numberwithin{equation}{section}

\newcommand{\rank}{{\rm rank}}

\newcommand{\sG}{{\mathcal G}}
\newcommand{\sH}{{\mathcal H}}

\newcommand{\sS}{{\mathcal S}}

\newcommand{\C}{{\mathbb C}}

\renewcommand{\P}{{\mathbb P}}
\newcommand{\Q}{{\mathbb Q}}
\newcommand{\R}{{\mathbb R}}

\newcommand{\Z}{{\mathbb Z}}

\newcommand{\id}{{\rm id\hspace{.1ex}}}

\title[Hyperbolic automorphism groups]{On K3 surfaces with hyperbolic automorphism groups} 

\author{Koji Fujiwara, Keiji Oguiso, Xun Yu}

\address{Okinawa Institute of Science and Technology Graduate University, Okinawa 904-0495, Japan; KUIAS and Research Institute for Mathematical Sciences, Kyoto University, Kyoto, Japan}
\email{koji.fujiwara@oist.jp}

\address{Mathematical Sciences, the University of Tokyo, Meguro Komaba 3-8-1, 
Tokyo, Japan, and National Center for Theoretical Sciences, 
Mathematics Division, National Taiwan University, 
Taipei, Taiwan}
\email{oguiso@ms.u-tokyo.ac.jp}

\address{Center for Applied Mathematics and KL-AAGDM, Tianjin University, Weijin Road 92, Tianjin 300072, P.R. China}
\email{xunyu@tju.edu.cn}

\thanks{Koji Fujiwara is partially supported by JSPS Grant-in-Aid (A) 25H00588. Keiji Oguiso is partially supported by JSPS Grant-in-Aid (A) 25H00587, 25K21992 and NCTS scholar program. Xun Yu is partially supported by NSFC (No. 12071337).}
\dedicatory{}


\begin{document}

\maketitle

\begin{abstract}
We show the finiteness of the N\'eron-Severi lattices of complex projective K3 surfaces whose automorphism groups are non-elementary hyperbolic with explicit descriptions, under the assumption that the Picard number $\ge 6$ which is optimal to ensure the finiteness. Our proof of finiteness is based on the study of genus one fibrations on K3 surfaces and recent work of Kikuta and Takatsu. 
\end{abstract} 

\section{Introduction}\label{sect0}

Since the global Torelli theorem for K3 surfaces by Piatetskii-Shapiro and Shafarevich \cite{PS71}, the automorphism groups of K3 surfaces have caught much attention from several aspects by several mathematicians. In this paper, inspired by a pioneering work by Kikuta \cite{Ki24}, and independently by Takatsu \cite{Ta24}, we address a new study of the automorphism group of a K3 surface from hyperbolicity in geometric group theory, with an unexpected relation with genus one fibrations on a K3 surface and the entropy (Theorems \ref{main13}, \ref{main}, \ref{main18}).

\medskip

Unless stated otherwise, by a {\it K3 surface}, we mean a smooth projective complex surface with no global $1$-form, i.e., $H^0(S, \Omega_S^1) = 0$, and with a nowhere vanishing global holomorphic $2$-form $\omega_S$, which is necessarily unique up to $\C^{\times}$. Let
$${\rm NS}\,(S) := {\rm Im}\, (c_1 : {\rm Pic}\,(S) \to H^2(S, \Z)) \simeq {\rm Pic}\,(S).$$
Then ${\rm NS}\, (S)$ is a free $\Z$-module of finite rank with the intersection form $(*, **)_S$, or equivalently, the integral symmetric bilinear form induced from the cup product of $H^2(S, \Z)$. We call ${\rm NS}\, (S) := ({\rm NS}\, (S), (*,**)_S)$ the {\it N\'eron-Severi lattice} of $S$ and $\rho\, (S) := {\rm rank}({\rm NS}\, (S))$ the {\it Picard number} of $S$. Then $1 \le \rho\, (S) \le 20$ and ${\rm NS}\, (S)$ is an even hyperbolic lattice. Here and hereafter we mean, by a {\it lattice}, a pair $L = (L, (*,**))$ of free $\Z$-module $L$ of finite rank with a symmetric bilinear form $(*, **) : L \times L \to \Z$. A lattice $L = (L, (*,**))$ of rank $r$ is called {\it even} (resp. {\it hyperbolic}) if $(v^2) \in 2\Z$ for all $v \in L$ (resp. $r \ge 2$ and $(*,**)$ is of signature $(1, r -1)$). We define an isomorphism of lattices and a sublattice of a lattice in natural manners. We denote the automorphism group of a lattice $L $ by ${\rm O}\, (L)$.

\medskip 

Gromov introduced the notion of hyperbolicity for a geodesic
metric space in his seminal paper, \cite{Gr87}. 
We call a group $G$ with a finite set of generators $S$ {\it hyperbolic} if the Cayley graph of $(G, S)$, which is considered as a metric space by the graph metric, is a hyperbolic space in the sense of Gromov. The hyperbolicity of $G$ does not depend on the choice of $S$. We call a hyperbolic group {\it non-elementary} if $G$ is not virtually cyclic, i.e., there is no cyclic subgroup $G'$, with $[G:G'] < \infty$. 
A virtually cyclic group is hyperbolic, and it is called {\it elementary}. Standard examples of hyperbolic groups include free groups of finite rank,
the fundamental groups of closed Riemannian manifolds of negative sectional
curvature, such as closed hyperbolic manifolds and closed orientable 
surfaces of genus at least two. Additionally, not all but some of Coxeter groups and Artin groups are hyperbolic. 

\medskip

Hyperbolic groups have been a central object in geometric group theory, as they provide a unified framework for studying a broad class of finitely generated groups. We recall some basic properties of hyperbolic groups. A hyperbolic
group does not contain a subgroup isomorphic to $\Z^2$.
More generally, any finitely generated solvable subgroup in a hyperbolic 
group is virtually cyclic. 
A hyperbolic group is always finitely presented. For finitely presented groups,
one can define
an isoperimetric inequality,  and 
a hyperbolic group satisfies a linear isoperimetric inequality.
This is one of the characterizations of hyperbolic groups, though we will not use this notion and this important fact in this paper. (See \cite[Chapter III, H2]{BH99}. For instance, the definition is Definition 2.1 and the equivalence follows from Proposition 2.7 and Theorem 2.9 there.)

\medskip

Max Dehn formulated three central  problems in combinatorial group theory in 1911:
the word problem, the conjugacy problem, the isomorphism problem. They are 
all solvable for hyperbolic groups in general. A hyperbolic group satisfies Tits alternative, i.e., for any finitely generated subgroup
$H$ 
in a hyperbolic group, either $H$ is virtually cyclic, or $H$ contains 
the free group of rank two, i.e., $\Z * \Z$. This implies that a non-elementary 
hyperbolic group has exponential growth. A hyperbolic group has finite virtual cohomological dimension if it is
virtually torsion free. 
It is unknown if all hyperbolic groups are virtually torsion free.
It is also unknown if all hyperbolic groups are residually finite.
A hyperbolic group $G$ has only finitely many finite subgroups
up to conjugation in $G$. 

\medskip

 Another characterization of a hyperbolic group 
is that a group is hyperbolic if it admits a properly discontinuous, co-compact
isometric action on a hyperbolic space in the sense of Gromov. 
Along this line, the 
hyperbolicity of a group has been generalized to relatively hyperbolic groups.
This notion plays an important role in this paper.  Many properties of hyperbolic groups have been extended to relatively hyperbolic groups with 
appropriate  modifications. Isometric actions by groups on hyperbolic spaces have been used
to study groups. A notable example is the theory by Masur-Minsky, \cite{MM99},
for the actions of the mapping class group of a compact surface
on its curve graph, which is hyperbolic in the sense of Gromov.
This approach gives a new proof of the solution of conjugacy 
problem of a mapping class group, \cite{MM99},  and 
the finiteness of the asymptotic dimension of a mapping class group, \cite{BBF15}.

\medskip

Let us return back to K3 surfaces. Let ${\rm Aut}\, (S)$ be the group of biregular automorphisms of a K3 surface $S$. We note that ${\rm Aut}\, (S)$ is always discrete, as $T_S \simeq \Omega_S^1$ by $\omega_S$ and thus $H^0(S, T_S) \simeq H^0(S, \Omega_S^1) = 0$. More strongly, ${\rm Aut}\, (S)$ is a finitely generated group by Sterk (\cite[Proposition 2.2]{St85}). We define
$${\rm Aut}\, (S)^* := {\rm Im}\, ({\rm Aut}\, (S) \to {\rm O}\, ({\rm NS}\, (S))\,\, ;\,\, f \mapsto f^*).$$

\medskip

The aim of this paper is to study K3 surfaces $S$ with non-elementary hyperbolic ${\rm Aut}\,(S)$ in terms of their N\'eron-Severi lattices ${\rm NS}\, (S)$ and to classify them under the assumption that the Picard number $\rho(S) \ge 6$ (Theorem \ref{main} (3), Theorem \ref{main18}; see also Theorem \ref{main12} (1) for a justification). Here we note that the isomorphism classes of ${\rm NS}\, (S)$ of K3 surfaces $S$ with non-elementary hyperbolic ${\rm Aut}\, (S)$ are no longer finite when $\rho(S) \le 5$ (Remark \ref{rem11}). 

\medskip

Besides an important observation by Kikuta and Takatsu (\cite{Ki24}, \cite{Ta24}), our proof is based on the study of genus one fibrations and their Mordell-Weil groups as well as their Jacobian fibrations (Theorems \ref{main} and \ref{main15}). We briefly recall here some notions needed to state the results (see \cite{CDL25} for details). Let $W$ be a smooth projective complex surface. A surjective morphism $\varphi : W \to C$ onto a smooth projective curve $C$ with connected fibers is called a {\it genus one fibration} if it is relatively minimal and a general fiber of $\varphi$ is a smooth genus one curve, i.e. an elliptic curve without specifying the origin. Any genus one fibration $\varphi : W \to C$ is then the minimal compactification of some torsor of the smooth locus of $j_{\varphi}$ of the unique genus one fibration $j_{\varphi} : W_{\varphi} \to C$ with global section, called the {\it Jacobian fibration} of $\varphi$. If $\varphi : W \to C$ has already a global section, then $W_{\varphi} = W$ and $j_{\varphi} = \varphi$ and by abuse of language, we also call a genus one fibration with global section a Jacobian fibration. When $W = S$ is a K3 surface, we know that $C = \P^1$ and $S_{\varphi}$ is again a K3 surface, which is not isomorphic to $S$ in general (\cite[Theorem 4.3.20 and Proposition 4.3.14]{CDL25}). We also note that a K3 surface $S$ has a genus one fibration (resp. a Jacobian fibration) if and only if ${\rm NS}\, (S)$ represents $0$ (resp. ${\rm NS}\, (S)$ contains $U$ which is the even unimodular hyperbolic lattice of rank $2$); see Proposition \ref{prop21} (\cite[Section 10, $2^\circ$]{Ni81}).

\medskip

Let $\varphi : S \to \P^1$ be a genus one fibration on a K3 surface $S$. Let $f \in {\rm Aut}\, (S)$. We call $f$ a translation automorphism of $\varphi$ if $f$ acts on each smooth fiber $S_t$ ($t \in \P^1$) of $\varphi$ as a translation of $S_t$. Translation automorphisms of $\varphi$ form a finitely generated abelian group of rank $\le \rho(S) -2$, called the {\it Mordell-Weil group} of $\varphi$. We denote by ${\rm MW}(\varphi)$ the Mordell-Weil group of $\varphi$ and by $mr(\varphi)$ the rank of ${\rm MW}(\varphi)$. One of the particular feature of a K3 surface $S$ is that $S$ often has several different genus one fibrations and Jacobian fibrations especially when $\rho(S)$ is larger. The following two invariants are crucial in our study:
$$mr(S) := {\rm Max}\, \{mr(\varphi)\,|\, \varphi : S \to \P^1 \text{ is a genus one fibration on } S\},$$
$$jmr (S) := {\rm Max}\, \{mr(\varphi)\,|\, \varphi : S \to \P^1 \text{ is a Jacobian fibration on } S\}.$$
We do not define $mr(S)$ (resp. $jmr(S)$) when $S$ has no genus one fibration (resp. no Jacobian fibration).  

\medskip

Our starting points are the following two observations:

\begin{theorem}\label{main11} Let $S$ be a K3 surface. Then the following four statements are equivalent:

\begin{enumerate}

\item ${\rm Aut}\,(S)$ is non-elementary hyperbolic; 

\item $\Z * \Z \subset {\rm Aut}\,(S)$ and $\Z^2 \not\subset {\rm Aut}\,(S)$; 

\item ${\rm Aut}\,(S)^*$ is non-elementary hyperbolic;

\item $\Z * \Z \subset {\rm Aut}\,(S)^*$ and $\Z^2 \not\subset {\rm Aut}\,(S)^*$. 

\end{enumerate}
Moreover, in this case, $\rho\, (S) \ge 3$. 
\end{theorem}

Theorem \ref{main11} is an immediate consequence of an important result of Kikuta and Takatsu (\cite[Theorem A (1)]{Ki24}, \cite[Theorem 5.13]{Ta24}), which claims the geometrical finiteness of ${\rm Aut}\,(S)$ of any K3 surface $S$, combined with general results on hyperbolic groups and the Torelli theorem for K3 surfaces \cite{PS71}. Since this is crucial also in our study of hyperbolicity, we mention this Theorem explicitly within our setting and present a proof following \cite{Ki24} in Section \ref{sect3}.

\begin{theorem}\label{main12} Let $S_i$ ($i = 1, 2$) be K3 surfaces such that ${\rm NS}\, (S_1) \simeq {\rm NS}\, (S_2)$. Then: 

\begin{enumerate} 

\item ${\rm Aut}\,(S_1)$ is non-elementary hyperbolic if and only if so is ${\rm Aut}\, (S_2)$. 

\item $S_1$ has a genus one fibration if and only if so does $S_2$ and in this case $mr(S_1) = mr(S_2)$.

\item $S_1$ has a Jacobian fibration if and only if so does $S_2$ and in this case $jmr(S_1) = jmr(S_2)$.

\end{enumerate}
\end{theorem}

Theorem \ref{main12} is then a formal but important consequence, which confirms that the non-elementary hyperbolicity of ${\rm Aut}\, (S)$ of a K3 surface depends only on the isomorphism class of the lattice ${\rm NS}\, (S)$. We prove Theorems \ref{main11} and \ref{main12} in Section \ref{sect2}, after a brief review of basic notions on K3 surfaces and lattices in Section \ref{sect1} and a brief introduction of hyperbolic groups in Section \ref{sect2}.

\medskip

We then give a more algebro-geometric characterization:

\begin{theorem}\label{main13} Let $S$ be a K3 surface with $\rho\, (S) \ge 5$. Then ${\rm Aut}\,(S)$ is non-elementary hyperbolic if and only if $mr(S) \le 1$ and ${\rm Aut}\,(S)$ is of positive entropy, in the sense that there is $f \in {\rm Aut}\, (S)$ such that $d_1(f) > 1$, where $d_1(f)$ is the spectral radius of $f^{*}|_{{\rm NS}\, (S)}$.
\end{theorem}

The same relation between relative hyperbolicity and positivity of entropy is already noticed by \cite[Corollary 4.5]{Ki24}. Also, the relation between hyperbolicity and the Mordell-Weil rank is one of the cores of our study. We deduce this from Proposition \ref{prop41}, which is a purely algebro-geometric statement. One of the referees kindly informed us that Proposition \ref{prop41} can be also deduced from Kleinian group theory and the corresponding result in Kleinian group theory should be known 
for some experts. In this paper, we keep our purely algebro-geometruc proof, as Proposition \ref{prop41} is very important in our study and other potential applications in algebraic geometry, whereas it does not seem to be much recognized in algebraic geometry. We also believe that our purely algebro-geometric proof is of its own interest.
 
Here by the assumption $\rho\, (S) \ge 5$, $S$ has at least one genus one fibration (Proposition \ref{prop21}). We will also use the assumption $\rho\, (S) \ge 5$, instead of $\rho\, (S) \ge 3$ in Theorem \ref{main11}, to ensure the equivalence between the fact that the entropy of ${\rm Aut}\,(S)$ is positive and the fact that ${\rm Aut}\,(S)$ is not virtually abelian (see \cite{Yu25} and proof of Theorem \ref{main13}). The property of positive entropy of ${\rm Aut}\,(S)$ is also crucial to connect our study with a fundamental work due to the third author \cite{Yu25}, and independently Brandhorst-Mezzedimi \cite{BM22}, on the complete classification of N\'eron-Severi lattices of K3 surfaces with zero entropy automorphism group ${\rm Aut}\, (S)$, to make a complete list in Section \ref{sect6}, see also Theorem \ref{main18}. We prove Theorem \ref{main13} in Section \ref{sect2} also by using Proposition \ref{prop41} in an essential way.

\medskip

By Theorem \ref{main12}, we may speak of the following sets:
\medskip
$$\sH_{\{0\}} := \{ {\rm NS}\, (S)\,|\, S \in \sS_{\{0\}}\text{ s.t. } mr(S) \le 1\}/{\rm isomorphisms};$$
$$\sH_{\{0\}}^{\ge r} := \{ {\rm NS}\, (S)\,|\, S \in \sS_{\{0\}} \text{ s.t. } \rho\, (S) \ge r\, ,\, mr(S) \le 1 \}/{\rm isomorphisms};$$
$$\sH_{U} := \{ {\rm NS}\, (S)\,|\, S \in \sS_{U} \text{ s.t. } mr(S) \le 1\}/{\rm isomorphisms};$$
$$\sH_{U}^{\ge r} := \{ {\rm NS}\, (S)\,|\, S \in \sS_{U} \text{ s.t. } \rho\, (S) \ge r\, ,\, mr(S) \le 1\}/{\rm isomorphisms};$$
$$\sH_{{\rm hyp}}^{\ge r} := \{ {\rm NS}\, (S)\,|\, S \in \sS\, \text{ s.t. } \rho\, (S) \ge r\, ,\, {\rm Aut}\, (S)\, \text{ is non-elementary hyperbolic}\}/{\rm isomorphisms},$$

\medskip

\noindent
where $\sS$ is the set of isomorphism classes of K3 surfaces and $\sS_{\{0\}}$ (resp. $\sS_{U}$) is the set of isomorphism classes of K3 surfaces with at least one genus one fibration (resp. at least one Jacobian fibration). Note that $S \in \sS_{\{0\}}$ (resp. $S \in \sS_{U}$) if and only if ${\rm NS}\, (S)$ represents $0$ (resp. $U \subset {\rm NS}\, (S)$) (Proposition \ref{prop21}).

\medskip

Our first main result is the following:

\begin{theorem}\label{main} Let $S$ be a K3 surface. Then
\begin{enumerate}

\item $\sH_{U}^{\ge 6}$ is a finite set.
 
\item $\sH_{0}^{\ge 6}$ is a finite set. 

\item $\sH_{{\rm hyp}}^{\ge 6}$ is a finite set, that is, the isomorphism classes of the N\'eron-Severi lattices ${\rm NS}\, (S)$ of K3 surfaces $S$ of $\rho(S) \ge 6$ with non-elementary hyperbolic ${\rm Aut}\, (S)$ are finite. 

\end{enumerate}
\end{theorem}

\begin{remark}\label{rem11} 
\begin{enumerate}
\item Theorems \ref{main} and \ref{main18} are also a kind of generalization of Nikulin \cite{Ni81} for $mr(S)=0$ to $mr(S)\le 1$. 
\item By Nikulin \cite[Theorem 0.2.6]{Ni81} with \cite[Theorem 1.3 and Appendix A]{Yu25}, \cite[Corollary 1.5]{BM22} and \cite[Section 16]{Ro22}, we see that $\sH_{{\rm hyp}}^{\ge 5}$ is an {\it infinite} set. In particular, our restriction $r \ge 6$ is optimal to make $\sH_{{\rm hyp}}^{\ge r}$ {\it finite}. 
\item By \cite[Theorems 0.2.4, 0.2.5, 0.2.7]{Ni81}, if a K3 surface $S$ satisfies both $\rho(S)\ge 6$ and $mr(S)=0$, then ${\rm Aut}(S)$ is finite. Thus, we have $mr(S)=1$ for K3 surfaces $S$ with ${\rm NS}(S)\in \sH_{{\rm hyp}}^{\ge 6}$. 
\item Mukai proposed the conjecture that for any K3 surface $S \in \sS_{\{0\}}$, the virtual cohomological dimension ${\rm vcd}({\rm Aut}(S))$ of ${\rm Aut}(S)$ is equal to $mr(S)$ (see e.g. \cite[Conjecture 1]{Ta24}). However, this conjecture seems to need some modification suggested by the examples in \cite[Theorem 0.2.6]{Ni81}; for example, if ${\rm NS}(S)\cong \langle 2^5\rangle \oplus D_4$, then one has $mr(S)=0$ and $|{\rm Aut(S)}|=\infty$, thus ${\rm vcd}({\rm Aut}(S))\ge 1$. 
\end{enumerate}
\end{remark}

\begin{remark}
\begin{enumerate}
\item For any non-projective K3 surface $X$, ${\rm Aut}(X)$ is virtually abelian (\cite[Theorem 1.5]{Og08}) so that ${\rm Aut}(X)$ is not non-elementary hyperbolic. 
\item For any projective K3 surface $S$ with $\rho(S)\le 2$, ${\rm Aut}\, (S)$ is virtually cyclic (see e.g., Theorem \ref{thm21}) so that ${\rm Aut}(S)$ is not non-elementary hyperbolic either.
\end{enumerate}
\end{remark}

In the course of proof of Theorem \ref{main}, we also need the following theorem, which has its own interest (see also Corollary \ref{cor51}):

\begin{theorem}\label{main15} Let $S$ be a K3 surface such that $U \subset {\rm NS}\, (S)$, i.e., a K3 surface with at least one Jacobian fibration. Then $mr(S) = jmr(S)$. 
\end{theorem}

We prove Theorem \ref{main15} in Section \ref{sect3} and Theorem \ref{main} in Section \ref{sect4}. Our proof of Theorem \ref{main15} uses an important result of Keum \cite[Lemma 2.1]{Ke00} (see also Theorem \ref{thm51}). We prove Theorem \ref{main} by dividing into two cases: (1) $L$ is of the form $U\oplus M$ and (2) $L\neq U\oplus M$.   We reduce case (2) to  case (1). In both cases, we use Theorem \ref{main15} and some useful method (Lemma \ref{lem:subtest}) in the lattice theory developed by the third author \cite[Theorem 4.1]{Yu25}. 

\medskip

In Section \ref{sect6}, we give a complete explicit list of $\sH_{{\rm hyp}}^{\ge 6}$.

\begin{theorem}\label{main18}
The set $\sH_{{\rm hyp}}^{\ge 6}$ consists of exactly the $166$ lattices listed in Table \ref{tab:166}. Moreover, among these $166$ lattices, exactly $145$ lattices contain $U$.
\end{theorem}

Roughly speaking, we first find finitely many explicit candidates $L$ and then check whether $L\in \sH_{{\rm hyp}}^{\ge 6}$ one by one. Slightly more precisely, as in the proof of Theorem \ref{main}, we divide into two cases (1) and (2) as above. We make our proof of Theorem \ref{main} much more explicit with help of computer algebra.

\begin{question}
\begin{enumerate}
\item Are there any intrinsic reasons why $\sH_{{\rm hyp}}^{\ge 6}$ is finite but $\sH_{{\rm hyp}}^{\ge 5}$ is infinite?

\item Determine K3 surfaces $S$ with ${\rm NS}(S)\in \sH_{{\rm hyp}}^{\ge 6}$ such that ${\rm Aut}(S)$ is virtually free. For instance, ${\rm Aut}(S)$ is virtually free of rank $11$ and $4$ if ${\rm NS}(S)$ is No. $144$ and $145$ in Table \ref{tab:166} respectively (\cite[Theorems 3.3 and 2.4]{Vi83}) (Note that No. $144$ and $145$ are exactly the lattices of rank $20$ in Table \ref{tab:166}).

\item Find explicit descriptions for the group ${\rm Aut}(S)$ (up to finite groups) for K3 surfaces $S$ with ${\rm NS}(S)\in \sH_{{\rm hyp}}^{\ge 6}$, especially for non-virtually free ones.
\end{enumerate}
\end{question}

\medskip

{\bf Acknowledgement.} We would like to thank Professor Shing-Tung Yau for invitation to ICBS 2025 at which the first version of this paper is completed. We would like to express our thanks to Professors Igor Dolgachev, Kouhei Kikuta, Shigeru Mukai, Taiki Takatsu and the referees for very careful reading, many suggestions and valuable remarks. The first author thanks Martin Bridson for insightful discussions concerning this work.

\section{Some basic properties of lattices and K3 surfaces} \label{sect1}

We recall a few very basic notions and facts on lattices, K3 surfaces and genus one fibrations on K3 surfaces, which we will frequently use in this paper. For more advanced or specific properties, we will introduce them when we will really need them. 

We use $U$ to denote the lattice with Gram matrix $\small{\left(\begin{array}{cc} 
0 & 1 \\
1 & 0 
\end{array} \right)}$. $U$ is the unique even unimodular hyperbolic lattice of rank $2$ up to isomorphisms. For an even negative definite lattice $M$, we denote by $M_{rt}$ {\it the root sublattice} of $M$, i.e., $M_{rt}$ is generated by the {\it roots} in $M$ which are elements $x$ in $M$ with $(x^2)=-2$. By $A_i$ ($i\ge 1$), $D_j$ ($j\ge 4$), $E_k$ ($k=6,7,8$) we mean the negative definite root lattice whose basis is given by the corresponding Dynkin diagram. If $L$ is a sublattice of a lattice $L'$ of the same rank, then we say $L'$ is an {\it overlattice} of $L$. Note that the index $[L':L]$ is equal to the order of the finite abelian group $L'/L$. We also note that $[L' :L]^2 = |{\rm det}\, L|/ |{\rm det}\, L'|$ by the elementary divisor theorem.

\medskip

Let $L$ be a nondegenerate even lattice. We denote the determinant of the Gram matrix $((e_i,e_j))$ of $L$ by ${\rm det}\, L$. Here $\{e_i\}$ is a $\Z$-basis of $L$ and ${\rm det}\, L$ does not depend on the choice of $\Z$-basis of $L$. We may view the dual $L^*={\rm Hom}(L,\Z)$ of $L$ as a subset of $L\otimes \Q$ via the canonical embedding $L\hookrightarrow L^*$ determined by the bilinear form of $L$. We use $l(L)$ to denote the minimum number of generators of the finite abelian group $L^*/L$. We call an element $0 \not= v \in L$ {\it isotropic} (resp. {\it primitive}) if $(v^2) = 0$ (resp. if the quotient $\Z$-module $L/\Z v$ is free).

\begin{definition}\label{def21} Let $L$ be an even hyperbolic lattice. We say $L$ is {\it isotropic} if $L$ has a nonzero element $x$ with $(x^2)=0$ (i.e., $x$ is an isotropic element in $L$). For a primitive isotropic element $x\in L$, $\langle x\rangle^\perp/\langle x\rangle = (\Z x)_{L}^{\perp}/\Z x$ is a negative definite lattice. For an isotropic even hyperbolic lattice $L$ and a primitive isotropic element 
$x \in L$, 
we define
$$mr(L, x) := {\rm rank}(\langle x\rangle^\perp/\langle x\rangle)-{\rm rank}((\langle x\rangle^\perp/\langle x\rangle)_{rt});$$
$$mr(L)= {\rm max}\, \{mr(L, x)\, |\; x \text{ is a primitive isotropic element in }L\}.$$ 
\end{definition}

\begin{remark}\label{overlattice}
\begin{enumerate}
\item If $L^\prime$ is an even overlattice of $L$, then we have $mr(L)\ge mr(L^\prime)$. The reason is as follows. Let $e' \in L'$ be a primitive isotropic element of $L'$ such that $mr(L') = mr(L', e')$. There is then a primitive isotropic element $e \in L$ such that $e = ae'$ for some positive integer $a$. Then we have a natural inclusion 
$$(\Z e)_{L}^{\perp}/\Z e \subset (\Z e')_{L'}^{\perp}/\Z e'\text{ with }{\rm rank}((\Z e)_{L}^{\perp}/\Z e)={\rm rank}((\Z e')_{L'}^{\perp}/\Z e').$$
Thus 
$$((\Z e)_{L}^{\perp}/\Z e)_{rt} \subset ((\Z e')_{L'}^{\perp}/\Z e')_{rt}.$$ 
Hence $mr(L)\ge mr(L, e) \ge mr(L', e')$ by the formula in Definition \ref{def21}.
\item Let $M$ be an even negative definite lattice. By definition, we have $$mr(L\oplus M)\ge mr(L).$$  By \cite[Theorem 10.2.1]{Ni81}, we have $mr(U\oplus M)=0$ if and only if the quotient group ${\rm O}(U\oplus M)/W(U\oplus M)$ is finite. Here $W(U\oplus M)$ is the Weyl group of $U\oplus M$ generated by the $(-2)$-reflections (as explained below). Note that we have $mr(S)=mr({\rm NS}(S))$ (Theorem \ref{thm22}).
\end{enumerate}
\end{remark}

\medskip

We call a group $G$ {\it virtually $\mathcal{V}$} if $G$ has a finite index subgroup which has the property $\mathcal{V}$. For example, we call $G$ virtually abelian if $G$ has a finite index abelian subgroup. 

\medskip

Let $S$ be a K3 surface. Let $r \in {\rm NS}\, (S)$ be a root. Then $x \mapsto x + (x, r)r$ defines an orthogonal involution of ${\rm NS}\, (S)$, called a $(-2)$-reflection. The Weyl group $W({\rm NS}\, (S))$ of ${\rm NS}(S)$ is the subgroup of ${\rm O}\, ({\rm NS}\, (S))$ generated by the $(-2)$-reflections. The Weyl group is a normal subgroup of ${\rm O}\, ({\rm NS}\, (S))$. Then, by the Torelli theorem for K3 surfaces \cite[Section 7, Theorem 1]{PS71}, we have

\begin{theorem}\label{thm21}The group homomorphism ${\rm Aut}\, (S) \to {\rm O}\, ({\rm NS}\, (S))/W({\rm NS}\, (S))$, induced by ${\rm Aut}\, (S) \to {\rm O}\,({\rm NS}\, (S))$, has a finite kernel and a finite cokernel. Moreover, ${\rm Aut}\, (S)$ is virtually cyclic, i.e., virtually abelian of rank $\le 1$ if $\rho\, (S) \le 2$. \end{theorem}

\begin{proof} As this is well-known, we just sketch a proof. That the cokernel is finite is a part of the global Torelli theorem for a K3 surface. That the kernel is finite follows from Fujiki-Lieberman's result \cite[Theorem 4.8]{Fu78} together with the fact that $H^0(S, T_S) = 0$. Recall that ${\rm Aut}\, (S)$ preserves the nef cone and its boundary. Moreover, the boundary consists of two half lines when $\rho\, (S) = 2$. From this, as well-known, it readily follows that ${\rm Aut}\, (S)^*$ is finite if $\rho\, (S) = 1$ and ${\rm Aut}\, (S)^*$ is virtually cyclic if $\rho\, (S) = 2$ (see e.g. \cite[Corollary 10.1.3]{Ni81} and \cite[Corollary 1]{GLP10}).
\end{proof}

The following two standard properties will be also frequently used:

\begin{proposition}\label{prop21}Let $S$ be a K3 surface. Then:
\begin{enumerate}
\item $S$ has a genus one fibration if and only if ${\rm NS}(S)$ is isotropic. In particular, $S$ has a genus one fibration if $\rho (S) \ge 5$.
\item $S$ has a Jacobian fibration if and only if $U \subset {\rm NS}\, (S)$.
\end{enumerate}
\end{proposition}

\begin{proof} For (1), see \cite[Section 3, Corollary 3]{PS71} and \cite[Chapter IV, Section 3.1, Corollary 2]{Ser12}. For (2), see \cite[Section 10, Theorem 10.2.1]{Ni81}.
\end{proof} 

\begin{theorem}\label{thm22} Let $S$ be a K3 surface with a genus one fibration $\varphi : S \to \P^1$. Let $f \in {\rm NS}\, (S)$ be the class of general fibers of $\varphi$. Then $f \not= 0$ is isotropic, primitive and satisfies
$$mr(\varphi) = mr({\rm NS}\, (S), f).$$
\end{theorem}

\begin{proof} The first assertion is clear and the second assertion follows from the Riemann-Roch formula for a K3 surface, which implies that $\dim |nv| = n$ for any nef primitive isotropic element $v$ and any positive integer $n$. The last formula is a direct consequence of the Shioda-Tate formula for genus one fibration (see e.g. \cite[Proposition 4.3.2 and Formula (4.3.1)]{CDL25}): 
\begin{equation}\label{STformula}
mr(\varphi) = \rho(S) - 2 - \sum_{b \in B} (m(b) -1),
\end{equation}
where $B \subset \P^1$ is the set of critical values of $\varphi$ and $m(b)$ is the number of irreducible components of the fiber $S_b$ ($b \in B$). Indeed, we obtain a $\Z$-basis of $((\Z f)_{{\rm NS}\, (S)}^{\perp}/\Z f)_{rt}$ by choosing suitable $m(b)-1$ irreducible components from each fiber $S_b$ ($b \in B$), by Kodaira's classification of singular fibers of genus one fibration (see e.g. \cite[Theorem 4.1.4]{CDL25}). \end{proof}

\section{Hyperbolic groups and Proofs of Theorems \ref{main11} and \ref{main12}}\label{sect2}

\subsection{Hyperbolic groups}

We recall basic definitions concerning hyperbolic groups.
A standard reference is \cite[Chapter III. H]{BH99}.

The {\it Cayley graph}, $Cay(G,S)$, of a finitely generated group $G$
with respect to a finite generating set $S$ is the metric graph whose vertices
are in one-to-one correspondence with the elements in $G$,
and which has an edge of length one joining $g$ to $gs$ for each $g \in G$ and $s \in S$.

Let $(X_1,d_1), (X_2,d_2)$ be metric spaces. A map $f:X_1\to X_2$
is called a {\it quasi-isometric embedding} if there exist
constants $\lambda \ge 1$ and $\epsilon \ge 0$ such that 
for all $x,y \in X_1$, 
$$\frac{d_1(x,y)}{\lambda} - \epsilon \le d_2(f(x),f(y)) \le \lambda d_1(x,y) + \epsilon.$$
If in addition, there exists a constant $C\ge 0$ such that 
every point of $X_2$ lies in the $C$-neighborhood of the image of $f$, then $f$ is called
a {\it quasi-isometry}. 
When such a map exists, $X_1$ and $X_2$ are said to be {\it quasi-isometric}.

Let $G$ be a finitely generated group, and $S_1, S_2$ two finite generating
sets. Then the Cayley graphs $Cay(G,S_1)$ and $Cay(G,S_2)$, with respect 
to $S_1, S_2$, respectively, are quasi-isometric. 
Consequently, the notion of quasi-isometry between finitely generated groups is well-defined, without reference to particular generating sets.

Let $\delta \ge 0$. A geodesic triangle in a metric space is said to be 
$\delta$-{\it thin} if each of its side is contained in the $\delta$-neighborhood
of the union of the other two sides. A geodesic space $X$ is said to be {\it $\delta$-hyperbolic}
if every triangle in $X$ is $\delta$-thin. If $X$ is $\delta$-hyperbolic for some 
$\delta$, one says that $X$ is {\it hyperbolic} (in the sense of Gromov).

Being hyperbolic is an invariant of quasi-isometry among geodesic spaces.
Hence the following definition does not depend on the choice of generators (although 
the value of $\delta$ does). 

\begin{definition}
A finitely generated group is {\it hyperbolic} (in the sense of Gromov)
if its Cayley graph is a $\delta$-hyperbolic metric space for some $\delta \ge 0$. 
\end{definition}

A virtually cyclic group is hyperbolic, and it is called an {\it elementary} 
hyperbolic group.

Suppose $G_2$ is a subgroup of finite index in $G_1$. Then $G_2$ is finitely generated
if and only if $G_1$ is finitely generated. Moreover, in this case they are quasi-isometric.
Similarly, suppose that there exists a surjective homomorphism $G_1 \to G_2$
with a finite kernel. Then $G_1$ is finitely generated if and only if $G_2$
is finitely generated, and in this case, they are quasi-isometric.
Hence, in both situations, $G_1$ is hyperbolic if and only if $G_2$
is hyperbolic.

\subsection{Relatively hyperbolic groups}

Gromov introduced a useful generalization of hyperbolic groups, called relatively hyperbolic groups, motivated by non-uniform lattices in rank-one symmetric spaces, \cite[Section 8]{Gr87}. 
There exist several characterizations and equivalent definitions of relatively hyperbolic groups, 
for example \cite[Definition 1]{Bo12} and \cite[Definition 8.4]{DS05}. We state the following
one, \cite[Definition 3.3]{Hr10}, which is most directly related to geometrically finite Kleinian groups.

One can generalize some basic notions in the classical case, $\Bbb H^n$, to a $\delta$-hyperbolic space $X$ (See for instance \cite[Part III, H3]{BH99} and \cite[Section 3.1]{Hr10}). First, one can define the boundary $\partial X$. Then  any isometry of $X$ extends to a homeomorphism of $\partial X$. Suppose the group $G$ acts properly discontinuously by isometries on a proper $\delta$-hyperbolic space $X$. 
Here proper means that all closed balls are compact.  An element $g \in G$ is {\it hyperbolic} (loxodromic) if it has infinite order and fixes exactly two points of $\partial X$. A subgroup $P$ of $G$ is a {\it parabolic subgroup} if it is infinite and contains no hyperbolic element. A parabolic subgroup $P$ of $G$ has a unique fixed point in $\partial X$, called a {\it parabolic point}. 
The stabilizer of a parabolic point is always a maximal parabolic subgroup.
For each point in $\partial X$, one can define horoballs at the point.  

\begin{definition}\label{def32}
Let $G$ be a countable group and let $\sG$ be a set of subgroups of $G$. Suppose $G$ acts properly discontinuously on a proper $\delta$-hyperbolic space $X$ such that  
$\sG$ is a set of representatives of the conjugacy classes of maximal parabolic subgroups. Suppose also that there is a $G$-equivariant collection of disjoint horoballs centered at the 
parabolic points of $X$, with union $U$ open in $X$, such that the quotient of $X - U$
by the action of $G$ is compact.
Then $G$ is {\it (relatively) hyperbolic relative to} $\sG$. 

\end{definition}

We note that if the set $\sG$ is empty, then $G$
is hyperbolic since it acts on $X$ properly
discontinuously 
and co-compactly, and vice versa. Also, $G$ is relatively hyperbolic with respect to the set of subgroups $\mathcal{G}=\{G\}$ in general.



We quote one result concerning relatively hyperbolic groups, which we use in this paper. 

\begin{theorem}\label{thm32}(\cite[Corollary 1.14]{DS05}) Let $G$ be a hyperbolic group relative to the set $\sG = \{G_1, \dots, G_m\}$. Suppose each $G_i$ is hyperbolic 
relative to a collection $\{G_i^1, \dots, G_i^{n_i}\}$. Then $G$ is hyperbolic relative to the collection 
of all $G_i^j$.

In particular, let $\sG' \subset \sG$ be the subset of $\sG$ obtained by getting rid of all virtually cyclic groups in $\sG$. Then $G$ is hyperbolic relative to $\sG'$. 
As a special case, if $\sG$ consists of virtually cyclic groups, then $G$ is a hyperbolic group.

\end{theorem}

We explain the ``in particular'' part. For each $G_i \in \sG$ we define the associated collection to 
consist solely of  $G_i$ itself if $G_i$ is not virtually cyclic, and 
to be empty if $G_i$ is virtually cyclic. This choice satisfies the assumption of the theorem 
and gives the family $\sG'$. 
In the special case, 
$\sG'$ is empty, so that $G$
is hyperbolic. 

\subsection{Kleinian groups}
We recall some definitions about Kleinian 
groups, cf. \cite{Ka09} and \cite{Rat06}.
Let $\Bbb H^n$ be the real hyperbolic space
of dimension $n$. Let ${\rm Isom}(\Bbb H^n)$
denote the group of isometries of $\Bbb H^n$. 
A discrete subgroup $G$ of ${\rm Isom}(\Bbb H^n)$
is called a {\it Kleinian group}.

Note that the (ideal) boundary of $\Bbb H^n$ is homeomorphic to the sphere $\Bbb S^{n-1}$. 
The {\it limit set} $\Lambda(G) \subset {\Bbb S}^{n-1}$ of $G$ is the set
of accumulation points of some (any) orbit 
$G(z)$ with $z \in \Bbb H^n$. 
A Kleinian group is called  {\it elementary} if its
limit set is a finite set. A Kleinian group is 
elementary if and only if it is virtually abelian. 

Let $C(\Lambda(G))$ be the convex hull of 
$\Lambda(G)$, which is the smallest convex subset $H$ of $\Bbb H^n$ such that 
$\bar H \cap \Bbb S^{n-1}=\Lambda(G)$, where $\bar H$ is the closure of $H$ in $\Bbb H^n \cup
\Bbb S^{n-1}$. If $\Lambda(G)$ is a single point, we define its convex hull to be empty. 

For $\epsilon >0$, let $C_\epsilon(\Lambda(G))$ be the open $\epsilon$-neighborhood of 
$C(\Lambda(G))$ in $\Bbb H^n$. 
A finitely generated Kleinian group $G$ is {\it geometrically finite} if there exists an $\epsilon >0$
such that the hyperbolic volume of $C_\epsilon(\Lambda(G))/G$ is finite.

The $G$-conjugacy classes $[P]$ of maximal parabolic subgroups $P$ of a Kleinian group $G$
are called {\it cusp} subgroups of $G$. 
The cusp subgroups  are parabolic subgroups 
in the isometry group of $\Bbb H^n$, hence virtually abelian of finite rank.

The following theorem due to Kikuta and Takatsu is crucial in our study:

\begin{theorem}\label{thm31} (\cite[Theorem A (1) and Corollary B]{Ki24}, \cite[Theorem 5.13]{Ta24}) Let $S$ be a K3 surface. Then ${\rm Aut}\, (S)$ is a geometrically finite Kleinian group. In particular, ${\rm Aut}\, (S)$ is 
either virtually abelian or non-elementary relatively hyperbolic relative to virtually
abelian subgroups of finite rank.
\end{theorem}
The ``in particular'' part is stated in \cite[Corollary B]{Ki24}, where
the collection of subgroups is not explicitly stated, but
for a geometrically finite Kleinian group, each of them fixes a point in the 
boundary of the hyperbolic space, so that it is virtually abelian of finite rank.

Combining Theorems \ref{thm32} and \ref{thm31}, we immediately have:
\begin{corollary}\label{cor35}
Let $S$ be a K3 surface. Then ${\rm Aut}\, (S)$ is 
either virtually abelian or relatively hyperbolic relative to a collection of virtually abelian subgroups of finite 
rank at least two.
Possibly, the collection is empty, in which case ${\rm Aut}\, (S)$ is hyperbolic. 

\end{corollary}

Another way to obtain this conclusion from \cite[Theorem A(1)]{Ki24} is to use \cite[Theorem 1.2.1]{HK05}, where
it is shown that a geometrically finite Kleinian group is  a relatively hyperbolic group 
with respect to a collection of virtually abelian subgroups of rank at least two,
using the idea of ``CAT(0) spaces with isolated flats''.
Their argument also relies on \cite{Bo93} and \cite{DS05}.

\subsection{Proofs of Theorems \ref{main11} and \ref{main12}}

We are ready to prove Theorems \ref{main11} and \ref{main12}.

\medskip
\noindent
{\it Proof of Theorem \ref{main11}.} Recall that ${\rm Ker}\, ({\rm Aut}\, (S) \to {\rm Aut}\, (S)^{*})$ is a finite group by Theorem \ref{thm21}. Then ${\rm Aut}\, (S)$ is non-elementary hyperbolic, if and only if so is ${\rm Aut}\, (S)^{*}$. This shows the equivalence (1) $\Leftrightarrow$ (3). 

Let us show that (1) $\Rightarrow$ (2). Let $G := {\rm Aut}\, (S)$. Since $G$ is non-elementary, it follows that $G$ is not a virtually abelian group of finite rank by the definition. By the Tits' alternative for $G$, which follows from the Tits' alternative for $G^* = {\rm Aut}\, (S)^*$ by \cite[Theorem 1.1]{Og06}, we have $\Z * \Z \subset G^*$ and thus $\Z * \Z \subset G$ as well. 

Let us show that (2) $\Rightarrow$ (1). Let $G := {\rm Aut}\, (S)$. Since $\Z * \Z \subset G$, it follows that $G$ is not virtually abelian. Since $\Z^2 \not\subset G$, it follows that virtually abelian subgroups of $G$ are only virtually cyclic subgroups. Since $G$ is relatively hyperbolic with respect to virtually abelian subgroups of finite rank, by Corollary  \ref{cor35}, $G$ is relatively hyperbolic with respect to the empty set. Thus $G$ is a non-elementary hyperbolic group. This proves (2) $\Rightarrow$ (1).

One can apply exactly the same argument in (1) $\Leftrightarrow$ (2) for $G := {\rm Aut}\, (S)^{*}$, to obtain the equivalence (3) $\Leftrightarrow$ (4). 

This completes the proof of Theorem \ref{main11}.

\medskip
\noindent
{\it Proof of Theorem \ref{main12}.} If ${\rm NS}\, (S_1) \simeq {\rm NS}\, (S_2)$, then both ${\rm Aut}\, (S_1)^{*}$ and ${\rm Aut}\, (S_2)^{*}$ are isomorphic to ${\rm O}\, ({\rm NS}\, (S_1))/W({\rm NS}\, (S_1))$ up to finite kernel and finite cokernel by Theorem \ref{thm21}. Hence ${\rm Aut}\, (S_1)^{*}$ is non-elementary hyperbolic if and only if so is ${\rm Aut}\, (S_2)^{*}$. This implies the assertion of Theorem \ref{main12} (1).

By Proposition \ref{prop21} and Theorem \ref{thm22}, both $mr(S)$ and $jmr(S)$ depend only on the isomorphism classes of the lattice ${\rm NS}\, (S)$. This implies the assertions (2) and (3). 

This completes the proof of Theorem \ref{main12}.  

\section{Proof of Theorem \ref{main13}.}\label{sect3}

In this section, we prove Theorem \ref{main13}. We first show the following purely-algebro geometric proposition, which is one of the cores of our proof of Theorem \ref{main13}. We believe that this proposition may have other applications:

\begin{proposition}\label{prop41} Let $X$ be a K3 surface and let $r$ be a positive integer such that $r \ge 2$. Assume that ${\rm Aut}\, (X)$ has a subgroup $G \simeq \Z^r$. Then there is a genus one fibration $\varphi : X \to \P^1$ and a finite index subgroup $G_0 \subset G$ such that $G_0 \subset MW (\varphi)$ (hence $G_0 \simeq \Z^r$ as well). In particular, $mr(X) \ge 2$.
\end{proposition}

\begin{proof} Set $L := {\rm NS}\, (X)$ and $\overline{A}(L) \subset L_{\R}:=L\otimes_\Z \R$ be the nef cone, i.e., the closure of the ample cone $A(L) \subset L_{\R}$. Note that $\overline{A}(L)$ is a strict convex closed cone with open interior $A(L)$ and $\overline{A}(L)$ is preserved by ${\rm Aut}\, (X)$. First we show the following:

\begin{lemma}\label{lem41} Let $g \in G$. Then $d_1(g) = 1$. 
\end{lemma}

\begin{proof} Assuming to the contrary that there is $g \in G$ with $d_1(g) >1$, we shall derive a contradiction. By \cite[Theorem 3.2]{Mc02}, $d_1(g)$ is a Salem number and by the Perron-Frobenius type theorem \cite{Bi67}, there is a unique $0 \not= v \in \overline{A}(L)$ up to $\R_{>0}$ such that $g(v) = d_1(g)v$. We have $(v^2) =0$ as 
$$(v^2) = (g(v)^2) = d_1(g)^2(v^2)$$
and $d_1(g) >1$. Note that $d_1(g^{-1}) = d_1(g)$ as the minimal polynomial $\Phi_g(t) \in \Z[t]$ of a Salem number is reciprocal. Then there exists a unique element $0 \not= u \in \overline{A}(L)$ up to $\R_{>0}$ such that $g^{-1}(u) = d_1(g)u$, i.e., $g(u) = d_1(g)^{-1}u$. We have $(u^2) = 0$ and $u$ and $v$ are linearly independent over $\R$ as $d_1(g) \not= d_1(g)^{-1}$ if $d_1(g) >1$. Then $(u, v) >0$ as $L_{\R}$ is of signature $(1,\rho\,(S) -1)$. 

\medskip

We choose $g \in G$ such that $d_1(g) >1$ is the minimum among all $d_1(g) > 1$ with $g \in G$. Since the degree of $\Phi_g(t) \in \Z[t]$ is bounded by $\rho\, (S)$, the minimum exists (\cite[Pages 229-230]{Mc02}). Let $h \in G$. Then, since $g \circ h = h \circ g$ it follows that 
$$g(h(v)) = d_1(g)h(v)\,\, ,\,\, g(h(u)) = d_1(g)^{-1}h(u)$$
and $h(v), h(u) \in \overline{A}(L)$. Thus by the uniqueness of $u$ and $v$ up to $\R_{>0}$, it follows that there are positive real numbers $a(h)$ and $b(h)$ such that 
$$h(v) = a(h)v\,\, ,\,\, h(u) = b(h)u.$$
By $(v,u) = (h(v), h(u)) = a(h)b(h)(v, u)$ and by $(v, u) > 0$, it follows that $a(h)b(h) =1$. 
We have then a group homomorphism
$$c : G \to \R_{>0} \times \R_{>0}\,\, ;\,\, h \mapsto (a(h), b(h) = a(h)^{-1}).$$
This $c$ is injective. Indeed, if $a(h) = b(h) = 1$, then as $\R\langle v, u \rangle^{\perp}$ is of negative definite and $h$-stable, it follows that $h$ is of finite order. Then $h = \id_X$, as $G$ has no torsion except $\id_X$.

\medskip

As $d_1(g) > 1$ is the minimum, there is a unique integer $n(h)$ such that $d_1(h) = d_1(g)^{n(h)}$. Since $c$ is injective, it follows that $h = g^{n(h)}$, a contradiction to the assumption that $G \simeq \Z^r$ with $r \ge 2$.  This completes the proof of Lemma \ref{lem41}. 
\end{proof}

\begin{lemma}\label{lem42} There is $0 \not= e \in \overline{A}(L) \cap L$ such that $e$ is primitive, isotropic and satisfies $g(e) = e$ for all $g \in G$. 
\end{lemma}

\begin{proof} Our argument closely follows an argument in \cite[Lemma 7.6]{OZ22}. By Lemma \ref{lem41}, $d_1(g) = 1$ for all $g \in G$. Choose $\id_X \not= h \in G$. Then $h$ is of infinite order. Since the characteristic polynomial of $h$ is the product of cyclotomic polynomials by $d_1(h)= 1$ (see \cite[Theorem 3.2]{Mc02}), there is a positive integer $m$ such that $h^m$ is unipotent and preserves $L$. Replacing $h$ by $h^m$, we may and will assume that $h$ is unipotent. Note that $h$ remains to be of infinite order. Thus $h - \id_X$ is nonzero and nilpotent also on $L$, that is, there is an endomorphism $0 \not= N : L \to L$ such that
$$h = \id + N\,\, ,\,\, N^s \not= 0\,\, ,\,\, N^{s+1} = 0$$
for some integer $s$ with $1 \le s \le \rho(S)$. 
Then 
$$h^n = \id + C_1(n)N + \ldots + C_{s}(n)N^s$$
where $C_k(n)$ is a binomial coefficient, which is a polynomial of $n$ of degree $k$ with rational coefficients.  
Since $A(L)$ is open in $L_{\R}$ and $L \cap A(L)$ generates $L_{\R}$ as an $\R$-vector space, there is $w \in A(L) \cap L$ such that $N^s(w) \not= 0$ as $N^s \not= 0$. Then 
$$\lim_{n \to \infty} \frac{h^n(w)}{n^s} = \lim_{n \to \infty}\frac{C_{s}(n)}{n^s}N^s(w) = \frac{N^s(w)}{s!}$$
Thus $e := N^s(w)$ satisfies, after a suitable positive multiple, that $e \in \overline{A}(L) \cap L$, $e \not=0$ is primitive, $h(e) = e$ and 
$$(e^2) = \lim_{n \to \infty} ((\frac{s! h^n(w)}{n^s})^2) = \lim_{n \to \infty} \frac{(s!)^2(w^2)}{n^{2s}} = 0,$$
as claimed.

\medskip

By \cite[Theorem 2.1]{Og07} and by the fact that $e \in L$ is primitive and $d_1(g) =1$ for all $g \in G$, it follows that $e$ is unique and satisfies $g(e) = e$ for all $g \in G$. This completes the proof of Lemma \ref{lem42}. 
\end{proof}

The class $e$ in Lemma \ref{lem42} gives a genus one fibration $\varphi : X \to \P^1$ preserved by $G \simeq \Z^2$. Since $\varphi$ has at least three singular fibers by \cite[Theorem 0.1]{VZ01}, ${\rm Im} (\iota : G \to {\rm Aut}\, (\P^1))$ is finite. Thus $G_1 := {\rm Ker}\, (\iota)$ is a finite index subgroup of $G$. Since any non-translation automorphism of an elliptic curve is of finite order $\le 6$, it follows that there is a finite index subgroup of $G_0 \subset G_1$ such that $G_0 \subset MW(\varphi)$. This completes the proof of Proposition \ref{prop41}. \end{proof}

We are now ready to prove Theorem \ref{main13}.

\medskip
\noindent
{\it Proof of Theorem \ref{main13}.} Assume first that ${\rm Aut}\, (S)$ is non-elementary and hyperbolic. Then ${\rm Aut}\, (S)$ is an infinite group and it is not virtually abelian, as ${\rm Aut}\, (S)$ is non-elementary. Thus by the contra-position of \cite[Theorem 1.4]{Og07}, we deduce that ${\rm Aut}\, (S)$ is of positive entropy. Since ${\rm Aut}\, (S)$ is hyperbolic, we have $\Z^2 \not\subset {\rm Aut}\, (S)$ by Theorem \ref{main11}. Thus $mr (S) \le 1$ when $S$ has a genus one fibration. Note that we do not need the assumption $\rho\, (S) \ge 5$ in this direction.

Let us show the converse. Since ${\rm Aut}\, (S)$ is of positive entropy, it follows that ${\rm Aut}\, (S)$ is an infinite group. For a K3 surface $X$ with $\rho\, (X) \ge 5$, the following three statements are equivalent by \cite[Theorem 9.1.1]{Ni81} and \cite[Theorem 1.4]{Og07} as it is pointed out by \cite[Introduction]{Yu25}:

\begin{enumerate}

\item ${\rm Aut}\, (X)$ is not of positive entropy, i.e. $d_1(f) = 1$ for all $f \in {\rm Aut}\, (X)$;

\item ${\rm Aut}\, (X)$ preserves a unique genus one fibration on $X$;

\item ${\rm Aut}\, (X)$ is virtually abelian (which is equivalent to that ${\rm Aut}\, (X)$ is almost abelian in \cite[Theorem 1.4]{Og07}, see e.g. \cite[Theorem 2.2 (2)]{DLOZ22} and its proof).

\end{enumerate}

Hence ${\rm Aut}\, (S)$ is not virtually abelian and thus $\Z * \Z \subset {\rm Aut}\, (S)$ by the Tits' alternative for a K3 surface \cite[Theorem 1.1]{Og06}. In addition, by our assumption $mr (S) \le 1$ and by Proposition \ref{prop41}, it follows that $\Z^r \not\subset {\rm Aut}\, (S)$. Hence ${\rm Aut}\, (S)$ is non-elementary hyperbolic by Theorem \ref{main11}. 

This completes the proof of Theorem \ref{main13}. 

\section{Proof of Theorem \ref{main15}.}\label{sect4}

In this section, we prove Theorem \ref{main15}. Let $L := (L, (**)_L)$ be an even nondegenerate lattice. The {\it discriminant group} $(A_L, q_L)$ of $L$ is the pair of finite abelian group $A_L := L^*/L$ and the quadratic form $q_{L} : A_L \to \Q/2\Z$ defined by $q_{L}(x\, {\rm mod}\, L) = (x^2)\, {\rm mod}\, 2\Z$. A (necessarily even nondegenerate) lattice $L' := (L', (*,**)_{L'})$ is called {\it of the same genus} as $L$ if $L'$ satisfies one of the following three equivalent conditions (\cite[Corollary 1.9.4 and Corollary 1.13.4]{Ni80}):

\begin{enumerate}

\item one has $(L', (*,**)_{L'}) \otimes_{\Z} \Z_p \simeq (L, (*,**)_{L}) \otimes_{\Z} \Z_p$ for all prime numbers $p$ and $(L', (*,**)_{L'}) \otimes_{\Z} \R \simeq (L, (*,**)_{L}) \otimes_{\Z} \R$ over the real number field $\R$. Here $\Z_p$ is the ring of $p$-adic integers. 

\item $(A_{L'}, q_{L'}) \simeq (A_{L}, q_{L})$ and $(*,**)_{L'}$ has the same signature as $(*,**)_{L}$.

\item $U\oplus L'\simeq U\oplus L$ as lattices.

\end{enumerate}

By $\sG(L)$, we denote the set of isomorphism classes of lattices of the same genus as $L$. 

We first prove the following purely algebraic:

\begin{proposition}\label{prop51}
Let $L$ be an even nondegenerate lattice and let $L \subset \overline{L}$ be an even overlattice of $L$. Let $\overline{L'} \in \sG(\overline{L})$. Then there is a sublattice $L' \subset \overline{L'}$ such that $L' \in \sG(L)$, and $\overline{L'}$ is thus an overlattice of $L'$.  
\end{proposition}

\begin{proof} Set $n := |\overline{L}/L| = [\overline{L}:L]$. Consider the sublattice of $\overline{L}$ defined by
$$n\overline{L} := \{nx\,|\, x \in \overline{L}\} \subset \overline{L}.$$
Then $L$ satisfies $n\overline{L} \subset L \subset \overline{L}$. By the first equivalent definition of the same genus, we have $\overline{L'} \in \sG(\overline{L})$ if and only if $n\overline{L'} \in \sG(n\overline{L})$. Thus, by the second equivalent definition, there is an isomorphism of discriminant groups
$$\phi : (A_{n\overline{L}}, q_{n\overline{L}}) \simeq (A_{n\overline{L'}}, q_{n\overline{L'}}).$$
Since $\overline{L}/n\overline{L} = \{x \in A_{n\overline{L}}\,|\, nx = 0\}$ and $\overline{L'}/n\overline{L'} = \{x \in A_{n\overline{L'}}\,|\, nx = 0\}$, it follows that we have $\overline{L}/n\overline{L} \simeq \overline{L'}/n\overline{L'}$ under $\phi$, preserving the induced quadratic forms. Since $L$ is an even overlattice of $n\overline{L}$, we have $q_{{n\overline{L}}}|_{L/n\overline{L}} = 0$ (see \cite[Proposition 1.4.1]{Ni80}). Hence $$q_{{n\overline{L'}}}|_{\phi(L/n\overline{L})} = 0.$$ There is then the lattice $L'$ such that  $n\overline{L'} \subset L' \subset \overline{L'}$, corresponding to $\phi(L/n\overline{L}) \subset A_{n\overline{L'}}$ (\cite[Proposition 1.4.1]{Ni80}). We have then $L' \in \sG(L)$. This is because the quadratic forms $q_{L'}$ and $q_{L}$ are the ones naturally induced from $q_{{n\overline{L'}}}$ and $q_{{n\overline{L}}}$ via $q_{{n\overline{L'}}}|_{\phi(L/n\overline{L})} = 0$ and $q_{{n\overline{L}}}|_{L/n\overline{L}} = 0$. This completes the proof of Proposition \ref{prop51}. \end{proof}

We recall the following very important result due to Keum \cite[Lemma 2.1]{Ke00} based on an idea of Fourier-Mukai transform by Mukai \cite[Theorem 1.5]{Mu87}:

\begin{theorem}\label{thm51}
Let $X$ be a K3 surface, let $\varphi : X \to \P^1$ be a genus one fibration with fiber class $f \in {\rm NS}\, (X)$ and $j_{\varphi} : X_{\varphi} \to \P^1$ the associated Jacobian fibration. Let $l$ be the positive integer such that  $(l) = \{(f, x)\,|\, x \in {\rm NS}\, (X)\}$ as ideals of $\Z$. Let $v \in {\rm NS}\, (X)$ such that $(f, v) = l$. Then $f/l \in {\rm NS}\, (X)^{*}$ and 
$${\rm NS}\, (X_{\varphi}) \simeq \Z (f/l) + {\rm NS}\, (X) \subset {\rm NS}\, (X)^{*},$$
as lattices.
\end{theorem}
It is well known that $X_{\varphi}$ is again a K3 surface \cite[Theorem 4.3.20 and Proposition 4.3.14]{CDL25}. Under this isomorphism, we may and will identify ${\rm NS}\, (X_{\varphi}) \simeq \Z (f/l) + {\rm NS}\, (X)$ and write
\begin{equation}\label{iota}
\iota : {\rm NS}\, (X) \hookrightarrow {\rm NS}\, (X_{\varphi}) = \Z (f/l) + {\rm NS}\, (X) \subset {\rm NS}\, (X)^{*}.
\end{equation}
Then we have $U \simeq \Z \langle f/l, v \rangle \subset {\rm NS}\, (X_{\varphi})$ and the orthogonal direct sum decomposition
$${\rm NS}\, (X_{\varphi}) = \Z \langle f/l, v \rangle \oplus (\Z \langle f/l, v \rangle)_{{\rm NS}\, (X_{\varphi})}^{\perp}.$$
We need the following direct consequence.
\begin{corollary}\label{cor51}
Let $X$ be a K3 surface and $\varphi : X \to \P^1$ be a genus one fibration with fiber class $f \in {\rm NS}\, (X)$ and $j_{\varphi} : X_{\varphi} \to \P^1$ the associated Jacobian fibration. Then $mr(X) \ge mr(X_{\varphi})$.
\end{corollary}

\begin{proof} Since ${\rm NS}\, (X_{\varphi})$ is an even over lattice of ${\rm NS}\, (X)$, the corollary follows from Remark \ref{overlattice} (1).    
\end{proof}

We are now ready to prove Theorem \ref{main15}.

\medskip
\noindent
{\it Proof of Theorem \ref{main15}.}
It suffices to show that $jmr(S) \ge mr(S)$. Set $X := S$ and let $\varphi : X \to \P^1$ be a genus one fibration with fiber class $f \in {\rm NS}\, (X)$ and let $j_{\varphi} : X_{\varphi} \to \P^1$ be the associated Jacobian fibration. We use the same notation as in Theorem \ref{thm51} and after that. We set $\overline{M'} := (\Z \langle f/l, v \rangle)_{{\rm NS}\, (X_{\varphi})}^{\perp}$. Then ${\rm NS}\, (X_{\varphi}) = U \oplus \overline{M'}$.

By the assumption, we have also an orthogonal decomposition ${\rm NS}\, (X) = U \oplus M$. Set $$\overline{M} := \iota(U)_{{\rm NS}\, (X_{\varphi})}^{\perp}$$ for this $U$ under the inclusion $\iota$ in equation \eqref{iota}. We have $\iota(M) \subset \overline{M}$ and $\overline{M'} \in \sG(\overline{M})$ by $U \oplus \overline{M'} = U \oplus \overline{M}$. By Proposition \ref{prop51}, there is then a sublattice $M' \subset \overline{M'}$ such that $M' \in \sG(M) = \sG(\iota(M))$. Then ${\rm NS}\, (X) = U \oplus M \simeq U \oplus M'$, and therefore, we have a Jacobian fibration $\varphi' : X \to \P^1$ such that $$M' \simeq \Z \langle f', s' \rangle_{{\rm NS}\, (X)}^{\perp},$$ where $f'$ is the fiber class of $\varphi'$ and $s'$ is a global section of $\varphi'$. Since $M' \subset \overline{M'}$, it follows that $M'_{rt} \subset (\overline{M'})_{rt}$. Hence, by Theorem \ref{thm51} and by the fact that $mr(\varphi) = mr(j_{\varphi})$ (see e.g. \cite[Proposition 4.3.2 and Definition 4.3.4]{CDL25}), it follows that 
$$mr(\varphi') = {\rm rank}\, M'  - {\rm rank}\, M'_{rt} \ge {\rm rank}\, \overline{M'}  - {\rm rank}\, \overline{M'}_{rt} = mr(j_{\varphi}) = mr(\varphi)$$
 by Remark \ref{overlattice} (1). Hence $jmr(S) \ge mr(S)$ as claimed.   

This completes the proof of Theorem \ref{main15}. 

\section{Proof of Theorem \ref{main}.}\label{sect5}

In this section, we prove Theorem \ref{main}. 

For a lattice $L$, we call a sublattice $M \subset L$ {\it primitive} if the quotient $\Z$-module $L/M$ is torsion free. The {\it primitive closure} $P$ of a sublattice $M$ in $L$ is the smallest primitive sublattice of $L$ containing $M$.

The following lemma, which is a modification of the third author's result \cite[Theorem 4.1]{Yu25}, is extremely useful in the proof of Theorem \ref{main}.  

\begin{lemma}\label{lem:subtest}
Let $L$ be an even isotropic lattice of rank at least $3$. Let $L_1\subset L$ be an isotropic primitive sublattice with ${\rm rank} (L_1)={\rm rank}(L)-1$. If $L$ and $L'$ are hyperbolic and
$$|\frac{{\rm det}(L)}{{\rm det}(L_1)}|>2,$$ 
then $$mr(L)\ge mr(L_1)+1.$$
\end{lemma}

\begin{proof}
Let $e\in L_1$ be a primitive isotropic element with $${\rm rank}(\langle e\rangle^\perp_{L_1}/\langle e\rangle)-{\rm rank}((\langle e\rangle^\perp_{L_1}/\langle e\rangle)_{rt})=mr(L_1).$$
Let $v\in L$ such that $\langle v\rangle =L_1^\perp\subset L$. Set $(v^2)=-a$, where $a$ is a positive even integer. Then there exists $k\in \Z$ with $k>0$ such that for each $\alpha\in L$ there are $x\in L_1$ and $m\in \Z$ such that $$\alpha =(\frac{x}{k},\frac{m v}{k})\in L_1^*\oplus \Z\langle v\rangle^* \text{ and } |\frac{{\rm det}(L)}{{\rm det}(L_1)}|=\frac{a}{k^2}$$ (see the proof of \cite[Theorem 4.1]{Yu25}). Suppose $(\alpha, e)=0$ and $(\alpha^2)=-2$. Then $(x, e)=0$ and $(x^2)\leq 0$. Thus, $$-2=(\alpha^2)=\frac{(x^2)}{k^2}+\frac{m^2 (v^2)}{k^2}\leq \frac{m^2 (v^2)}{k^2}=-\frac{m^2 a}{k^2}=-m^2 |\frac{{\rm det}(L)}{{\rm det}(L_1)}|.$$
By $|\frac{{\rm det}(L)}{{\rm det}(L_1)}|>2$, we have $m=0$. Thus, $\alpha\in L_1$. Then $e$ is a primitive isotropic element in $L$ with $${\rm rank}(\langle e\rangle^\perp_{L}/\langle e\rangle)-{\rm rank}((\langle e\rangle^\perp_{L}/\langle e\rangle)_{rt})=mr(L_1)+1.$$
Hence, $mr(L)\ge mr(L_1)+1$.
\end{proof}

Let $\mathcal{P}$  be the set of prime integers and $$\mathcal{P}_1:=\{p\in \mathcal{P} | \,\,\, p | {\rm det}(L) \text{ for some even hyperbolic lattice }L \text{ s.t. } {\rm rank}(L)\ge 5, \, mr(L)=0\}.$$

\begin{lemma}\label{finitePrime}
The set $\mathcal{P}_1$ is finite.
\end{lemma}

\begin{proof}
If $mr(L) = 0$, then as ${\rm rank}\, L \ge 5$, such $|{\rm det}\, L|$ are finitely many, except $2^m3^n$ for some nonnegative integers $m$, $n$ if ${\rm rank}(L) = 5$, by Nikulin \cite[Theorem 0.2.6]{Ni81}. 
\end{proof}

\medskip
\noindent
{\it Proof of Theorem \ref{main} (1).} Let $S$ be a K3 surface with $r:= \rho(S)-2 \ge 4$ and $$L := {\rm NS}\, (S) = U \oplus M.$$ Then $S$ admits a Jacobian fibration and $jmr(S) = mr(S)$ by Theorem \ref{main15}. Hence, by Nikulin \cite[Theorem 10.2.1]{Ni81}, $mr(S) = 0$ (i.e., $jmr(S) = 0$) if and only if ${\rm Aut}\, (S)$ is finite. The isomorphism classes of such ${\rm NS}\, (S)$ are also finite (\cite[Theorem 0.2.2]{Ni81}). So, we may and will assume that $mr(S) = 1$. Our goal is to show that the isomorphism classes of $L = U \oplus M$ are finite under additional assumption that $mr(S) = 1$.

\medskip

Since $jmr(S) = mr(S) =1$, it follows that $S$ has a Jacobian fibration $\varphi : S \to \P^1$ such that $\langle f, s \rangle_{{\rm NS}\, (S)}^{\perp} = M$ and $mr(\varphi) = 1$. Here $f$ is the fiber class and $s$ is the class of global section of $\varphi$.

Then $M$ is a negative definite even lattice of rank $r$ and thus $M_{rt}$ is a negative definite even lattice of rank $r-1$ by the Shioda-Tate formula 
(\ref{STformula}). Since $M_{rt}$ is a direct sum of $ADE$ lattices with bounded rank ($r\le 18 = 20 -2$), there are only finitely many isomorphism classes of such lattices $M_{rt}$. 

\medskip

Thus it suffices to show that there are only finitely many isomorphism classes of $M$ for each fixed $M_{rt}$. Let $P$ denote the primitive closure of $M_{rt}$ in $L$, which is the same as the primitive closure of $M_{rt}$ in $M$ as $M$ is primitive in $L$. Fix a $\Z$-basis $\langle v_i \rangle_{i=1}^{r-1}$ of the $P$ and the Gram matrix representation of $P =(a_{ij})_{1 \le i, j \le r-1}$. Then we extend them to a $\Z$-basis $\langle v_i \rangle_{i=1}^{r}$ of $M$ with the Gram matrix $(a_{ij})_{1 \le i, j \le r}$. 

Let $\langle x, y\rangle$ be a $\Z$-basis of $U$ such that $(x^2) = (y^2) = 0$ and $(x, y) = 1$. For a positive integer $p$, we define the sublattice $L_p$ of $L = U \oplus M$ of rank $r+1 = {\rm rank}\, L -1$ by
$$L_p := \Z \langle x, py -v_r, v_1, \cdots, v_{r-1}\rangle.$$
Then, by $(x^2) = (x, v_i) = 0$, $(x, py-v_r) = p$, we find that $L_p \subset L$ is primitive, isotropic and hyperbolic, and by the shape of the Gram matrices of $L_p$ and $P$, we have
$$|{\rm det}\, L_p| = p^2|{\rm det}\, P|,$$
whereas
$$|{\rm det}\, L| = |{\rm det}\, M|.$$
By Lemma \ref{finitePrime}, $mr(L_p) \ge 1$ for a sufficiently large prime integer $p$ regardless of $M$. We choose and fix such a $p$\footnote{By the explicit list of the  N\'eron-Severi lattices of K3 surfaces with finite automorphism group in \cite[Introduction]{Ni81} (see also \cite[Section 16, Table]{Ro22}), one can choose $p\le 11$ regardless of $r$. We don't need this fact in the proof of Theorem \ref{main}, but such an explicit choice of $p$ is crucial to obtain our explicit classification in Theorem \ref{main18}.}. Then, since $mr(L_p) = 1$ and $mr(L) =1$, by the contraposition of Lemma \ref{lem:subtest}, 
we have 
$$|{\rm det}\, M| = |{\rm det}\, L| \le 2p^2 |{\rm det}\, P|\le 2p^2 |{\rm det}\, M_{rt}|.$$
Since $M$ is negative definite of bounded rank, it follows that the isomorphism classes of such $M$ are finite. Thus so are the isomorphism classes of $L$. 

This completes the proof of Theorem \ref{main} (1).

\medskip
\noindent
{\it Proof of Theorem \ref{main} (2).} Let $X$ be a K3 surface with $\rho (X) \ge 6$ and with $mr(X) \le 1$. As before by combining the results of Nikulin, we may and will assume that $mr(X) = 1$ and may show the isomorphism classes of such ${\rm NS}\, (X)$ are finite. Let $\varphi : X \to \P^1$ be a genus one fibration such that $mr(\varphi) =1$. 

\medskip

{\it From now, we use the same notation as in Theorem \ref{thm51}.} For instance $j_{\varphi} : X_{\varphi} \to \P^1$ is the associated Jacobian fibration, and $f$ is the fiber class of $\varphi$ and $l$ is the minimal positive integer of $(f,x)$ for $x \in {\rm NS}\, (X)$. As in Theorem \ref{thm51}, we choose $v \in {\rm NS}\, (X)$ such that $(f, v) = l$. By Theorem \ref{thm51} and Corollary \ref{cor51}, we have $\rho(X_{\varphi}) = \rho (X) \ge 6$ and $mr(X_{\varphi}) \le mr(X) = 1$. 

\medskip

We have the orthogonal decomposition ${\rm NS}\, (X_{\varphi}) = U \oplus M$, where $U = \Z\langle f/l,v \rangle$ and $$M := (\Z \langle f/l, v \rangle)_{{\rm NS}\, (X_{\varphi})}^{\perp}.$$ Because of $\rho(X_{\varphi}) \ge 6$ and $mr(X_{\varphi}) \le 1$, the isomorphism classes of such $M$ are finite by Theorem \ref{main} (1). Since $[{\rm NS}(X_{\varphi}):{\rm NS}(X)]=l$, it suffices to bound $l$ uniformly. 

Note that the number of possible values $d := |{\rm det}\, M|$ is also finite. Since $$(\Z f)_{{\rm NS}\, (X)}^{\perp}/ \Z f \simeq M$$ by Theorem \ref{thm51}, we have natural primitive embeddings 
$$M \subset (\Z f)_{{\rm NS}\, (X)}^{\perp} \subset {\rm NS}\, (X).$$ 
Then ${\rm NS}\, (X)$ is an overlattice of $M_{{\rm NS}\, (X)}^{\perp} \oplus M$ with 
$$\frac{|{\rm det}\, M_{{\rm NS}\, (X)}^{\perp}| \cdot |{\rm det}\, M|}{|{\rm det}\, {\rm NS}\,(X)|} = |{\rm NS}\, (X)/(M_{{\rm NS}\, (X)}^{\perp} \oplus M)|^2 \le d^2$$
as the gluing of these two sublattices is along a subgroup of $M^*/M$ with $|M^*/M| = d$ (for gluing of lattices, see \cite[Section 1, $5^\circ$]{Ni80}, \cite[Section 2]{Mc11a}, \cite[Section 4]{Mc16}, \cite[Section 4]{OY20}).  On the other hand, since ${\rm NS}\,(X_{\varphi}) = U \oplus M$ and $[{\rm NS}\,(X_{\varphi}):{\rm NS}\,(X)] = l$, 
we have
$$|{\rm det}\, {\rm NS}\,(X)| = l^2d.$$
From these two formulas, we obtain that 
$$|{\rm det}\, M_{{\rm NS}\, (X)}^{\perp}| \le l^2d^2.$$
Note also that $M_{{\rm NS}\, (X)}^{\perp} = \Z \langle f, sv+ u\rangle$ for some positive integer $s$ and an element $u \in M$.\footnote{Indeed, we may take $w=v$ as one of the referees pointed out to us.} Set $$w := sv + u.$$ Then $w$ is primitive inside $M_{{\rm NS}\, (X)}^{\perp}$ and ${\rm NS}\, (X)$. By $$(f^2) = 0, \, (f, v) = l,\,  (f, u) = 0,$$ it follows that $|{\rm det}\, M_{{\rm NS}\, (X)}^{\perp}| = (ls)^2$. Thus 
$$d^2 \ge [{\rm NS}\,(X) : M_{{\rm NS}\, (X)}^{\perp} \oplus M]^2 = \frac{(ls)^2d}{l^2d} = s^2.$$
Hence $d \ge s \ge 1$. Thus $s$ is also bounded as so is $d$. Set $L := {\rm NS}\, (X)$ and consider $$z := w -af$$ where $a \in \Z$. Then $M_{{\rm NS}\, (X)}^{\perp} = \Z \langle f, z \rangle$ with $(z^2) =  (w^2) - a(2ls)$. Thus, we may and will choose $a$, hence $z$, so that 
$$2ls \le b:= (z^2) \le 4ls -1.$$ 

\medskip

Let us show the boundedness of $l >0$. Let $$z_1:=f+q z=(1-a q)f+q w,$$ where $q$ is a positive integer. Note that $z_1$ is a primitive element in $L$ with $(z_1^2)=2qsl+q^2 b$. Consider the primitive closure $P_1$ of the sublattice 
$$L_1 := \Z z_1 \oplus M \subset L.$$
 Then $|{\rm det}\, L_1| = (2q sl +q^2 b)d$ and $|{\rm det}\, P_1|=\frac{1}{h^2}|{\rm det}\, L_1|$, where $h$ is a factor of $d$. Note that $P_1$ is a primitive sublattice of $L$ with ${\rm rank}\, P_1 = {\rm rank}\, L -1$. Moreover, $P_1$ is also isotropic as $P_1$ is hyperbolic with ${\rank}\, P_1 \ge 5$. 

By the finiteness of possible values of $d$ and Lemma \ref{finitePrime}, as before, we may assume that $mr(P_1) \ge 1$ for any prime integer $q$ such that $q\notin \mathcal{P}_1$ and $q \nmid d$. We choose one of such $q$.

 Then, since $mr(L) = 1$, by the contraposition of Lemma \ref{lem:subtest}, we have
$$\frac{l^2}{2q sl +q^2 b} = \frac{l^2d}{(2q sl +q^2 b)d} = \frac{|{\rm det}\, L|}{|{\rm det}\, L_1|}= \frac{|{\rm det}\, L|}{h^2 |{\rm det}\, P_1|} \le \frac{2}{h^2}\le 2.$$
Combining this with previous inequalities $l>0$, $b \le 4ls -1$ and $s \le d$, we obtain
$$l^2 \le 2(2q sl +q^2 b)< 2(2q sl +4q^2 ls) \le 2(2q dl +4q^2 ld),$$
which implies
$$l\le 2(2q d +4q^2 d)\le (4q  +8q^2 )d.$$

This completes the proof of Theorem \ref{main} (2). 

\medskip
\noindent
{\it Proof of Theorem \ref{main} (3).} Let $S$ be a K3 surface of $\rho (S) \ge 6$ with non-elementary hyperbolic ${\rm Aut}\, (S)$. By $\rho (S) \ge 6$, $S$ has at least one genus one fibration. Since ${\rm Aut}\, (S)$ is hyperbolic, we have $mr(S) \le 1$ by Theorem \ref{main11}. Thus ${\rm NS}\, (S)$ is an element of the set $\sH_0^{\ge 6}$, which is a finite set by Theorem \ref{main} (2). This completes the proof of Theorem \ref{main} (3).
 
\section{Explicit list of lattices in $\sH_{{\rm hyp}}^{\ge 6}$}\label{sect6}

In this section, we prove Theorem \ref{main18}.  The Picard number $\rho=\rho(S)$ of any K3 surface $S$ is at most $20$. So we only need to determine $$\sH_{{\rm hyp}}^{\rho}:=\{L|\, L\in\sH_{{\rm hyp}}^{\ge \rho}, \, {\rm rank}(L)=\rho \}/{\rm isomorphisms},$$ where $6\le \rho\le 20$. Note that for $L\in \sH_{{\rm hyp}}^{\rho}$, we have $mr(L)\le 1$ if and only if $mr(L)=1$ by Remark \ref{rem11} (3). Since ${\rm NS}(S)$ is a primitive sublattice of the unimodular rank $22$ lattice $U^{\oplus 3}\oplus E_8^{\oplus 2}$, it follows that $\rho(S)+l({\rm NS}(S))\le 22$. Thus, by Theorem \ref{main13}, it suffices to classify all lattices $L$ satisfying the following conditions:
\begin{enumerate}
\item $L$ is an even hyperbolic lattice of positive entropy;
\item $mr(L)=1$ (see Remark \ref{rem11} (2));
\item $6\le {\rm rk}(L)\le 20$ and ${\rm rk}(L)+l(L)\le 22$.
\end{enumerate}

\medskip

 We classify such lattices inductively from rank $6$ to $20$. For each rank $r$, we can obtain all lattices satisfying (1)-(3) in two steps: 
\begin{enumerate}
\item[(I)]:  classify rank $r$ lattices $L$ satisfying (1)-(3) of the form $L=U\oplus M$;

\item[(II)]: classify rank $r$ lattices $L$ satisfying (1)-(3) but not containing $U$.
\end{enumerate}

\medskip

For Step (I), since $mr(U\oplus M)=1$, replacing $M$ by another member in the genus of $M$ if necessary, we may assume that $M_{rt}$ has rank $r-3$. Thus, the primitive closure $M_1$ of $M_{rt}$ in $M$  is an even overlattice of a root lattice. Then we have a primitive extension $M_{1}\oplus M_1^{\perp} \subset M=M_{1}\oplus_\phi M_1^{\perp}$ for some gluing map $\phi: H_1\rightarrow H_2$ (see e.g. \cite[Section 4]{Mc16}), where $H_1\subset M_{1}^*/M_1$, $H_2\subset (M_{1}^{\perp})^*/M_1^{\perp}$, $H_1\cong H_2\cong \Z/k\Z$ for some positive integer $k$ dividing   $$c_{M_{1}}:= \text{the largest order of the elements in } M_{1}^*/M_1$$ by ${\rm rank}(M_1^{\perp})=1$. Note that $U\oplus M_1$ is a rank $r-1$ primitive sublattice of $U^{\oplus 3}\oplus E_8^{\oplus 2}$. 

Thus, in Step (I), it suffices to do the following steps: 
 \begin{enumerate}
\item[(i)] compute the (finite) set $$\mathcal{M}_1:=\{M_1 |\,\, M_1 \text{ an even overlattice of a root lattice of rank }r-3 \text{ and }r+l(M_1)\le 23 \};$$ 
\item[(ii)] For each $M_1 \in \mathcal{M}_1$, let $c_{M_1}$ denote the largest order of the elements in $M_{1}^*/M_1$. Define $p_{M_1}=p$ as follows: if $mr(U\oplus M_1)>0$, we set $p=1$; if $mr(U\oplus M_1)=0$ and $r=6$, we set $p=5$; if $mr(U\oplus M_1)=0$ and $r>6$, let $p$ be the smallest positive integer such that $|p^2{\rm det}(M_1)|$ is not in the (finite) set $$\mathcal{D}_{r-1}:=\{d|\, d=|{\rm det} (L)| \text{ for some even hyperbolic lattice $L$ of rank $r-1$ and }mr(L)=0 \}.$$
By Nikulin's classification \cite{Ni81} (see also \cite[Section 16]{Ro22}), we have $\mathcal{D}_{6}=\{ 4, 5, 8, 9, 12, 16, 64, 81\}$, $\mathcal{D}_{7}=\{ 4, 6, 8, 10, 12,$ $16, 18, 32, 128\}$, $\mathcal{D}_{8}=\{ 3, 4, 7, 8, 12, 15,$ $16, 27, 64, 256\}$, $\mathcal{D}_{9}=\{ 2, 4, 6, 8, 12, 16, 32, 128, 512\}$, $\mathcal{D}_{10}=\{ 1, 4, 9, 16, 64, 256\}$, $\mathcal{D}_{11}=\{ 2, 8, 32, 128\}$, $\mathcal{D}_{12}=\{ 3, 4, 16, 64\}$, $\mathcal{D}_{13}=\{ 4, 8, 32\}$, $\mathcal{D}_{14}=\{ 4, 16\}$, $\mathcal{D}_{15}=\{ 8\}$, $\mathcal{D}_{16}=\{ 4\}$, $\mathcal{D}_{17}=\{ 2\}$, $\mathcal{D}_{18}=\{ 1\}$, $\mathcal{D}_{19}=\{ 2\}$.

\item[(iii)] Compute the (finite) set $\mathcal{E}_{M_1}$ of genera of primitive extensions $M$ of $M_1\oplus \langle -a\rangle$ such that $M$ is even and $$a\le 2p_{M_1}^2 k^2\le 2p_{M_1}^2 c_{M_1}^2,\,\, r+l(M)\le 22,\,\, mr(U\oplus M)=1,$$ where $k:=[M:M_1\oplus \langle -a\rangle].$ In fact, by Lemma \ref{lem:subtest}, the first inequality is a necessary condition for $mr(U\oplus M)=1$ since $M$ contains a primitive corank one sublattice of determinant $p^2\cdot |{\rm det}(M_1)|$ (see the proof of Theorem \ref{main} (1)). We slightly adapt \cite[Section 4, Genus test]{Yu25} to compute members $M^\prime$ in the genus of $M$; then by computing the rank of the root lattices $M^{\prime}_{rt}$, we check whether $mr(U\oplus M)=1$ or not.
\end{enumerate}

\medskip

Let $$\mathcal{L}^r_U:=(\bigcup_{M_1\in\mathcal{M}_1}\{U\oplus M |\; M\in \mathcal{E}_{M_1}\})/{\rm isomorphisms}.$$ Then the subset $$\{ L| \, L\in \mathcal{L}^r_U \text{ of positive entropy}\}/{\rm isomorphisms}$$  of $\mathcal{L}^r_U$ consists of all possible rank $r$ lattices which satisfy (1)-(3) and contain $U$. By \cite[Theorem 1.14.2]{Ni80}, the genus of every lattice in $\mathcal{L}^r_U$ contains only one isomorphism class.

Then, in Step (II), it suffices to do the following: For each $L=U\oplus M \in \mathcal{L}^r_U$, we compute all possible rank $r$ sublattices $L^\prime \subset L$ such that the following statements hold

\begin{enumerate}
\item[(${\rm i}^\prime$)] $mr(L^\prime)=1$;
\item[(${\rm ii}^\prime$)] $L=\Z (f/l) + L^\prime$ with $(\Z f)_{L^\prime}^{\perp}/ \Z f \simeq M$ for some primitive isotropic element $f$ in $L^\prime$, $l>1$ (see Theorem \ref{thm51});
\item[(${\rm iii}^\prime$)] $L^\prime$ is of positive entropy and has no sublattice isometric to $U$.
\end{enumerate}

\medskip
 For any integer $r\ge 6$, we define 
$$\mathcal{P}_r:=\{p\in \mathcal{P}\, | \, p^2| {\rm det}(L) \text{ for some rank $r$ even hyperbolic lattice} L \text{ s.t }mr(L)\le 1,\, U\subset L\}.$$ 
\begin{claim}\label{claim}
Let $L=U\oplus M \in \mathcal{L}^r_U$ and let $L^\prime\subset L$ be a sublattice of rank $r\ge 6$ and $l=[L:L']>1$. Suppose $L'$ satisfies the conditions (${\rm i}^\prime$)-(${\rm ii}^\prime$). Then:
\begin{enumerate}
\item If $p$ is a prime integer dividing $l$, then $p\in \mathcal{P}_r$.
\item If an integer $l_1>1$ divides $l$, then there is a rank $r$ sublattice $L_1\subset L$ satisfying $l_1=[L:L_1]$ and the conditions (${\rm i}^\prime$)-(${\rm ii}^\prime$).
\item If $L'$ satisfies the condition (${\rm iii}^\prime$), then $r-l(L)\le 4$.

\end{enumerate}
\end{claim}

\begin{proof}
Since $L'$ satisfies (${\rm ii}^\prime$), we have $p^2 | {\rm det}(L')=l^2 {\rm det}(L)$. By \cite[Lemma 3.7]{Yu25}, there is an even overlattice $L''$ of $L'$ such that $l(L'')\le 3$ and $p^2| {\rm det}(L'')$. Thus, $U\subset L''$ by \cite[Corollary 1.13.5]{Ni80} and $mr(L'')\le mr(L')=1$ by Remark \ref{overlattice} (1). Then $p\in \mathcal{P}_r$.

Let $k:=l/l_1$. By (${\rm ii}^\prime$), we have $L=\Z (f/l) + L^\prime$ with $(\Z f)_{L^\prime}^{\perp}/ \Z f \simeq M$ for some primitive isotropic element $f\in L^\prime$. In this case, $f_1:=f/k$ is a primitive isotropic element in $L_1:=\Z (f_1) + L^\prime\subset L$ such that $L=\Z (f_1/l_1) + L_1$ with $(\Z f_1)_{L_1}^{\perp}/ \Z f_1 \simeq M$. Thus $L_1$ satisfies (${\rm ii}^\prime$). Since $L_1$ is an even overlattice of $L'$, we have $mr(L_1)\le mr(L')=1$, which implies the assertion (2).

By (${\rm ii}^\prime$) again, we infer that $l(L^\prime)\le l(L)+2$. If $r-l(L)\ge 5$, then $r-l(L^\prime)\ge 3$, which implies  $U\subset L^\prime$ by \cite[Corollary 1.13.5]{Ni80}, a contradiction. Thus (3) holds.
\end{proof}

Now we explain how to compute all the lattices $L'\subset L\in \mathcal{L}^r_U$ satisfying (${\rm i}^\prime$)-(${\rm iii}^\prime$) in practice. For each prime $p\in \mathcal{P}_r$, we compute inductively the (finite) sets
$$\mathcal{L}_{(L,\, p^s)}':= \{ L'|\, L'\subset L\,  \text{ satisfies (${\rm i}^\prime$)-(${\rm ii}^\prime$) and the index }[L:L']=p^s\}/{\rm isomorphisms},\,s\ge 1.$$
By computing the discriminant group $(A_{L'}, q_{L'})$ and applying \cite[Theorem 1.14.2]{Ni80}, we conclude that for any lattice $L'$ in the sets $\mathcal{L}_{(L,\, p^s)}'$, its genus consists of a single isomorphism class.  By Claim \ref{claim} (2), if $\mathcal{L}_{(L,\, p^{s_0})}'$ is empty for some $s_0$, then $\mathcal{L}_{(L,\, p^s)}'$ is empty for all $s\ge s_0$. Such an $s_0$ exists by Theorem \ref{main} (2). If $L$ satisfies the property that $\mathcal{L}_{(L,\, p)}'$ is nonempty for at most one prime $p\in\mathcal{P}_r$, then by Claim \ref{claim} (1)-(2), we obtain all desired sublattices $L'\subset L$ by selecting, from the union $\cup_{1\le s<s_0}\mathcal{L}_{(L,\, p^s)}'$, precisely those lattices that satisfy (${\rm iii}^\prime$). It turns out that every lattice $L\in\mathcal{L}^r_U$ satisfies this property.

For computation in cases $6\le r\le 20$, we use McMullen's package \cite{Mc11b} and a mixture of Mathematica (\cite{Wo}), Magma (\cite{BCP}), PARI/GP (\cite{Th}), and SageMath (\cite{The}). All the computer algebra programs and the outputs for computing $\mathcal{L}^r_U$ and $L'$ satisfying (${\rm i}^\prime$)-(${\rm iii}^\prime$) can be found in the ancillary files to \cite{FOY25}. These files can be obtained at https://arxiv.org/src/2507.13726v1/anc. 

Next we explain how to get the result in Table \ref{tab:166} for two typical cases $r=6,9$. 

\subsection*{Case $r=6$} There are three rank $r-3=3$ root lattices $A_1^{\oplus 3}$, $A_2\oplus A_1$, $A_3$ and they do not have nontrivial even overlattices. Then we have $\mathcal{M}_1=\{A_1^{\oplus 3}, A_2\oplus A_1, A_3\}$. It follows that $(c_{M_1}, p_{M_1})=(2,5), (6, 5), (4,5)$ for $M_1=A_1^{\oplus 3}, A_2\oplus A_1, A_3$ respectively. By computing the required primitive extensions in Step (I)-(iii) for $M_1=A_1^{\oplus 3}$ with the help of computer, we find that $\mathcal{E}_{A_1^{\oplus 3}}$ consists of eleven genera represented by the following lattices $M$
$$
{\scriptsize
\left(
\begin{array}{cccc}
-2 & 0 & 0 & -1 \\
0 & -2 & 0 & -1 \\
0 & 0 & -2 & 1 \\
-1 & -1 & 1 & -4
\end{array}\right), \left(\begin{array}{cccc}
-2 & 0 & 0 & 0 \\
0 & -2 & 0 & -1 \\
0 & 0 & -2 & -1 \\
0 & -1 & -1 & -4
\end{array}\right), \left(\begin{array}{cccc}
-2 & 0 & 0 & 0 \\
0 & -2 & 0 & 0 \\
0 & 0 & -2 & 0 \\
0 & 0 & 0 & -4
\end{array}\right), \left(
\begin{array}{cccc}
-2 & 0 & 0 & -1 \\
0 & -2 & 0 & -1 \\
0 & 0 & -2 & 1 \\
-1 & -1 & 1 & -6
\end{array}\right),
}
$$

$$
{\scriptsize
\left(\begin{array}{cccc}
-2 & 0 & 0 & 0 \\
0 & -2 & 0 & 0 \\
0 & 0 & -2 & 1 \\
0 & 0 & 1 & -4
\end{array}\right), \left(\begin{array}{cccc}
-2 & 0 & 0 & 0 \\
0 & -2 & 0 & 0 \\
0 & 0 & -2 & 1 \\
0 & 0 & 1 & -6
\end{array}\right), \left(\begin{array}{cccc}
-2 & 0 & 0 & 0 \\
0 & -2 & 0 & 0 \\
0 & 0 & -2 & 0 \\
0 & 0 & 0 & -6
\end{array}\right), \left(\begin{array}{cccc}
-2 & 0 & 0 & 0 \\
0 & -2 & 0 & 0 \\
0 & 0 & -2 & 0 \\
0 & 0 & 0 & -8
\end{array}\right), 
}
$$

$$
{\scriptsize
\left(
\begin{array}{cccc}
-2 & 0 & 0 & -1 \\
0 & -2 & 0 & -1 \\
0 & 0 & -2 & 1 \\
-1 & -1 & 1 & -10
\end{array}\right), \left(\begin{array}{cccc}
-2 & 0 & 0 & 0 \\
0 & -2 & 0 & 0 \\
0 & 0 & -2 & 0 \\
0 & 0 & 0 & -10
\end{array}\right), \left(\begin{array}{cccc}
-2 & 0 & 0 & 0 \\
0 & -2 & 0 & 0 \\
0 & 0 & -2 & 0 \\
0 & 0 & 0 & -16
\end{array}\right).
}
$$
By taking direct sum $U\oplus M$ and comparing these lattices with those in \cite[Appendix]{Yu25}, we find that the first four genera $M$ correspond to four hyperbolic lattices $U\oplus M$ of zero entropy; the remaining genera $M$ correspond to seven hyperbolic lattices (No. 3, 9, 12-16 in Table \ref{tab:166}) of positive entropy. After doing similar calculations for $M_1=A_2\oplus A_1, A_3$, we find that $\mathcal{L}^6_U$ consists of $27$ lattices, and exactly $16$ lattices (No. 1-16 in Table \ref{tab:166}) of them are of positive entropy. Note that $ Det_{\mathcal{L}^6_U}:=\{|{\rm det}(L_1)|\,\, |\,\, L_1\in \mathcal{L}^6_U\}=\{ 12, 13, 16, 17, 20, 21, 24, 28, 29, 32, 36, 44, 45, 48, 64, 68, 80, 128\}.$

For each $L\in \mathcal{L}^6_U$, we compute the sublattices $L^\prime$ of $L$ such that the conditions (${\rm i}^\prime$)-(${\rm iii}^\prime$) hold. For example, let $L=U\oplus A_1^{\oplus 3}\oplus [-6]$ (No. 12 in Table \ref{tab:166}). Note that a prime number $p $ is in the finite set $\mathcal{P}_6$ if and only if $p^2$ divides the determinant of some lattice of rank $6$ in either $\mathcal{L}^6_U$ or the set $\mathcal{F}_U$ in \cite[Theorem 0.2.2]{Ni81}. From this, we have $\mathcal{P}_6=\{2, 3\}$. If $L^\prime\in\mathcal{L}_{(L,\, 3)}'$, then $|{\rm det}(L^{\prime})|=2^4\cdot 3^3$ and $L^{\prime}$ has an even overlattice $L^{\prime\prime}$ with $|{\rm det}(L^{\prime\prime})|=2^2 \cdot 3^3 $ by \cite[Lemma 3.7 (2)]{Yu25}. It follows from \cite[Corollary 1.13.5]{Ni80} and Remark \ref{overlattice} (1) that $L''$ contains $U$ and $mr(L'')\le mr(L')=1$, which contradicts to $2^2 \cdot 3^3 \notin \mathcal{D}_6 \cup Det_{\mathcal{L}^6_U}$. Thus, $\mathcal{L}_{(L,\, 3)}'$ is empty. The sublattices of $L=U\oplus A_1^{\oplus 3}\oplus [-6]$ satisfying the condition (${\rm ii}^\prime$) are isometric to lattices with a basis $(\alpha_1,\cdots ,\alpha_6)$ and Gram matrix of the form
$$
{\scriptsize
\left(\begin{array}{cccccc}
0 & l & 0 & 0 & 0 & 0 \\
l & a^\prime & b_1 & b_2 & b_3 & b_4 \\
0 & b_1 & -2 & 0 & 0 & 0 \\
0 & b_2 & 0 & -2 & 0 & 0 \\
0 & b_3 & 0 & 0 & -2 & 0 \\
0 & b_4 & 0 & 0 & 0 & -6
\end{array}\right)},$$
where $a^\prime\in\{0,2,\cdots,2l-2\}$, $b_i\in \{0,1,\cdots, l-1\}$ for all $i$. For $l=2$, such lattices form seven genera which are represented by the lattices $L'_i$ ($1\le i\le 7$) with $(a^\prime, b_1, b_2, b_3,b_4)$ given by 
$$(0,1,1,1,1), (0,0,1,1,0), (0,0,0,1,1), (0,0,0,1,0),(0,0,0,0,1),(2,1,1,1,1),(0,0,0,0,0).$$ 
Since $l(L'_i)=4$ ($1\le i\le 6$) and $L'_7=U(2)\oplus A_1^{\oplus 3}\oplus [-6]$, by \cite[Theorem 1.14.2]{Ni80}, each of the seven genera consists of a single isomorphism class. The first five lattices contain $U$ generated by the pairs $(\alpha_2,\alpha_3)$, $(\alpha_2,\alpha_4)$, $(\alpha_2,\alpha_5)$, $(\alpha_2,\alpha_5)$, $(\alpha_2,\alpha_6)$ respectively, and by randomly choosing $5$-tuples of elements in $L'_6$, we find that $L'_6$ contains a rank $5$ sublattice $L_1$ (generated by $\alpha_1+\alpha_6$, $\alpha_1-\alpha_2$, $\alpha_3$, $\alpha_4$, $\alpha_5$) of determinant $48$ and  $mr(L_1)>0$. Thus, $mr(L_i')>1$ ($1\le i\le 6$). By computation with the help of computer, $L_7'$ has exactly two nontrivial even overlattices $U\oplus A_1^{\oplus 3}\oplus [-6]$ (No. 12 in Table \ref{tab:166}) and $U\oplus A_1^{\oplus 2}\oplus A_2$ which is contained in the table in \cite[Theorem 0.2.2]{Ni81}. Since both of them are in $\sH_{U}^{6}$, we have that $mr(L_7')=1$. Clearly the lattice $L_7'$ is of positive entropy and contains no $U$. Thus, $L_7'$ is the only lattice in $\mathcal{L}_{(L,\, 2)}'$ and it satisfies (${\rm iii}^\prime$). Similarly, we find that $\mathcal{L}_{(L,\, 2^2)}'$ is empty. Hence,  by Claim \ref{claim} (1)-(2), all the sets $\mathcal{L}_{(L,\, 2^{s})}'$ ($s\ge 2$), $\mathcal{L}_{(L,\, 3^{s'})}'$ ($s'\ge 1$) are empty and $L_7'$ (No. 152 in Table \ref{tab:166}) is the only sublattice of $L$ satisfying all the conditions (${\rm i}^\prime$)-(${\rm iii}^\prime$).

As in the case $L=U\oplus A_1^{\oplus 3}\oplus [-6]$, we compute the desired $L^\prime $ for each $L\in \mathcal{L}^6_U$. It turns out that totally there are fourteen such $L^\prime$ (No. 146-159 in Table \ref{tab:166}). Thus, $\sH_{{\rm hyp}}^{6}$ consists of $16+14=30$ lattices.

\subsection*{Case $r=9$} Similar to Case $r=6$, we obtain that $\mathcal{L}^9_U$ consists of $37$ lattices. Moreover, exactly $25$ lattices (i.e., No. 60-84 in Table \ref{tab:166}) in $\mathcal{L}^9_U$ are of positive entropy. It turns out that there are exactly four lattices $L$ in $\mathcal{L}^9_U$ such that $9-l(L)\le 4$. The four lattices are as follows:
$$U(2)\oplus A_3\oplus D_4,\, U\oplus A_1^{\oplus 2}\oplus D_4\oplus [-4], \, U\oplus A_1^{\oplus 2}\oplus D_4\oplus [-6],\, U\oplus A_1^{\oplus 6}\oplus [-4].$$
Note that the first lattice is of zero entropy by \cite[Appendix]{Yu25} and the other three lattices are No. 82-84 in Table \ref{tab:166}. By similar calculations in Case $r=6$, none of them contains sublattices $L^\prime$ satisfying (${\rm i}^\prime$)-(${\rm iii}^\prime$). Thus, by Claim \ref{claim}, there is no lattice in $\sH_{{\rm hyp}}^{9}$ not containing $U$. Thus, $\sH_{{\rm hyp}}^{9}$ contains only $25$ lattices.

The other cases (i.e., $r =7,8, 10,11,\dots,20$) can be done similarly. Totally, there are exactly $166$ lattices, listed in Table \ref{tab:166}, satisfying the conditions (1)–(3), and they belong to $166$ distinct genera. All of them have primitive embeddings into the K3 lattice $U^{\oplus 3}\oplus E_8^{\oplus 2}$ by \cite[Theorem 1.12.2]{Ni80}. This completes the proof of Theorem \ref{main18}.

  \begin{table}[htp]
\caption{The $166$ lattices in $\sH_{{\rm hyp}}^{\ge 6}$ }\label{tab:166}
\begin{center}
{\scriptsize
\begin{tabular}{|c| p{3.8cm}|c|p{3.8cm}|c|p{3.8cm}|}
\hline 
 No & Lattice  & No & Lattice & No & Lattice \\
\hline 
$1$ & $U\oplus (A_3\oplus_4 [-84])$ & $57$ & $U\oplus A_1^{\oplus 2}\oplus A_3\oplus [-4]$ & $113$ & $U\oplus  D_4\oplus D_5$ \\ 
 \hline 
 $2$ & $U\oplus A_3\oplus [-6]$ & $58$ & $U\oplus A_1^{\oplus 5}\oplus [-4]$ &  $114$ & $U\oplus  A_1\oplus A_8$  \\ 
 \hline 
 $3$ & $U\oplus (A_3\oplus_2 [-28])$ & $59$ & $U\oplus A_1^{\oplus 5}\oplus [-6]$ &$115$ & $U\oplus  E_8\oplus [-24]$ \\ 
 \hline 
 $4$ & $U\oplus (A_3\oplus_4 [-116])$ & $60$ & $U\oplus E_6 \oplus [-4]$ & $116$ & $U\oplus  A_1\oplus A_2 \oplus D_6$ \\ 
 \hline 
 $5$ & $U\oplus A_1\oplus (A_2\oplus_3 [-48])$ & $61$ & $U\oplus (D_6\oplus_2 [-18])$ &  $117$ & $U\oplus  D_4^{\oplus 2}\oplus [-4]$ \\ 
 \hline 
 $6$ & $U\oplus A_3\oplus [-8]$ & $62$ & $U\oplus (E_6\oplus_3 [-60])$ & $118$ & $U\oplus  A_1^{\oplus 4} \oplus D_5$ \\ 
 \hline 
 $7$ & $U\oplus (A_3\oplus_2 [-36])$ & $63$ & $U\oplus A_1\oplus (D_5\oplus_4 [-44])$ &$119$ & $U\oplus   (D_9 \oplus_4 [-28])$  \\ 
 \hline 
 $8$ & $U\oplus (A_1\oplus A_2\oplus [-6])$ &  $64$ & $U\oplus D_6\oplus [-6]$ &$120$ & $U\oplus   A_{10}$  \\ 
 \hline 
 $9$ & $U\oplus A_1\oplus (A_2\oplus_3 [-66])$ & $65$ & $U\oplus (E_6\oplus_3 [-78])$ & $121$ & $U\oplus A_1\oplus  A_2 \oplus E_7$ \\ 
 \hline 
 $10$ & $U\oplus A_2\oplus (A_1\oplus_2 [-30])$ & $66$ & $U\oplus A_6\oplus [-4]$ & $122$ & $U\oplus A_1\oplus  E_8 \oplus [-8]$ \\ 
 \hline 
 $11$ & $U\oplus  A_3\oplus [-12])$ & $67$ & $U\oplus E_6\oplus [-10]$ & $123$ & $U\oplus A_1\oplus  A_9$\\ 
 \hline 
 $12$ & $U\oplus A_1^{\oplus 3}\oplus [-6]$ & $68$ & $U\oplus (A_6\oplus_7 [-224])$ & $124$ & $U\oplus A_1^{\oplus 3}\oplus  D_7$ \\ 
 \hline 
 $13$ & $U\oplus A_1^{\oplus 3}\oplus [-8]$ & $69$ & $U\oplus D_6\oplus [-8]$ & $125$ & $U\oplus A_4\oplus  E_7$ \\ 
 \hline 
 $14$ & $U\oplus (A_3\oplus_2 [-68])$ &  $70$ & $U\oplus A_1\oplus (A_5\oplus_3 [-24])$ & $126$ & $U\oplus A_2\oplus  D_9$ \\ 
 \hline 
 $15$ & $U\oplus A_1^{\oplus 3}\oplus [-10]$ & $71$ & $U\oplus A_1 \oplus D_5\oplus [-4]$ & $127$ & $U\oplus (D_{10}\oplus_2  [-14])$ \\ 
 \hline 
 $16$ & $U\oplus A_1^{\oplus 3}\oplus [-16]$ & $72$ & $U\oplus (A_6\oplus_7 [-252])$ & $128$ & $U\oplus A_1^{\oplus 2}\oplus  D_9$  \\ 
 \hline 
 $17$ & $U\oplus A_4\oplus [-4]$ & $73$ & $U\oplus ((A_1\oplus D_5)\oplus_2 [-18])$ &  $129$ & $U\oplus E_8\oplus  (A_2\oplus_3 [-48])$  \\ 
 \hline 
 $18$ & $U\oplus (A_4\oplus_5 [-120])$ & $74$ & $U\oplus (E_6\oplus_3 [-114])$ & $130$ & $U\oplus D_6\oplus  E_6$   \\ 
 \hline 
 $19$ & $U\oplus (A_4\oplus_5 [-130])$ & $75$ & $U\oplus A_1^{\oplus 3}\oplus A_4$ & $131$ & $U\oplus A_1^{\oplus 2}\oplus A_2\oplus  E_8$ \\ 
 \hline 
 $20$ & $U\oplus A_3\oplus (A_1\oplus_2 [-14])$ & $76$ & $U\oplus (A_6\oplus_7 [-308])$ & $132$ & $U\oplus A_{12}$  \\ 
 \hline 
  $21$ & $U\oplus A_4\oplus [-6]$ & $77$ & $U\oplus D_6\oplus [-12]$ & $133$ & $U\oplus D_5\oplus  D_7$ \\ 
 \hline 
 $22$ & $U\oplus D_4\oplus [-8]$ & $78$ & $U\oplus A_2\oplus D_4\oplus [-4]$ & $134$ & $U\oplus A_1\oplus A_4\oplus  E_8$   \\ 
 \hline 
 $23$ & $U\oplus A_1\oplus A_3\oplus [-4]$ & $79$ & $U\oplus E_6\oplus [-18]$ &  $135$ & $U\oplus D_7\oplus  E_6$   \\ 
 \hline 
 $24$ & $U\oplus (A_4\oplus_5 [-170])$ & $80$ & $U\oplus D_6\oplus [-14]$ & $136$ & $U\oplus A_6\oplus  E_8$ \\ 
 \hline 
 $25$ & $U\oplus (A_4\oplus_5 [-180])$ & $81$ & $U\oplus D_4 \oplus (A_2 \oplus_3 [-48])$ &  $137$ & $U\oplus A_1\oplus D_5\oplus E_8$  \\ 
 \hline 
 $26$ & $U\oplus A_1^{\oplus 2}\oplus (A_2\oplus_3 [-30])$ & $82$ & $U\oplus A_1^{\oplus 2}\oplus D_4\oplus [-4]$ & $138$ & $U\oplus A_1\oplus E_6\oplus E_8$  \\ 
 \hline 
 $27$ & $U\oplus D_4\oplus [-10]$ & $83$ & $U\oplus A_1^{\oplus 2}\oplus D_4\oplus [-6]$ & $139$ & $U\oplus A_7\oplus  E_8$ \\ 
 \hline 
 $28$ & $U\oplus A_1\oplus (A_3\oplus_4 [-84])$ & $84$ & $U\oplus A_1^{\oplus 6}\oplus [-4]$ & $140$ & $U\oplus (A_{14}\oplus_3  [-6])$  \\ 
 \hline 
 $29$ & $U\oplus A_1\oplus A_3\oplus [-6]$ & $85$ & $U\oplus A_2\oplus D_6$ & $141$ & $U\oplus E_{8}\oplus (E_7\oplus_2  [-10])$  \\ 
 \hline 
 $30$ & $U\oplus A_2^{\oplus 2}\oplus [-6]$ & $86$ & $U\oplus  (E_7\oplus_2 [-26])$ &  $142$ & $U\oplus A_1\oplus  D_7\oplus E_8$  \\ 
 \hline 
 $31$ & $U\oplus D_4\oplus [-14]$ & $87$ & $U\oplus  E_7\oplus [-8]$ & $143$ & $U\oplus D_9\oplus  E_8$ \\ 
 \hline 
 $32$ & $U\oplus A_1^{\oplus 2}\oplus (A_2\oplus_3 [-48])$ & $88$ & $U\oplus  (E_7\oplus_2 [-34])$ &  $144$ & $U\oplus A_2\oplus  E_8^{\oplus 2}$ \\ 
 \hline 
 $33$ & $U\oplus D_4\oplus [-16]$ & $89$ & $U\oplus  E_7\oplus [-10]$ & $145$ & $U\oplus A_1^{\oplus 2}\oplus  E_8^{\oplus 2}$ \\ 
 \hline 
 $34$ & $U\oplus A_1^{\oplus 4}\oplus [-4]$ & $90$ & $U\oplus  A_4\oplus D_4$ & $146$ & $U(2)\oplus A_3\oplus [-4]$ \\ 
 \hline 
 $35$ & $U\oplus A_1^{\oplus 4}\oplus [-6]$ & $91$ & $U\oplus  A_5\oplus A_3$ & $147$ & $U(2)\oplus (A_3\oplus_2 [-20])$    \\ 
 \hline 
 $36$ & $U\oplus A_1^{\oplus 4}\oplus [-8]$ & $92$ & $U\oplus  (D_7\oplus_4 [-100])$ & $148$ & $U(2)\oplus A_3\oplus [-8]$ \\ 
 \hline 
 $37$ & $U\oplus D_5\oplus [-4]$ & $93$ & $U\oplus  (E_7\oplus_2 [-50])$ & $149$ & $U(2)\oplus A_1^{\oplus 3}\oplus [-4]$ \\ 
 \hline 
 $38$ & $U\oplus (D_5\oplus_4 [-76])$ & $94$ & $U\oplus  E_7\oplus [-14]$ &  $150$ & $U(2)\oplus (A_3\oplus_2 [-36])$  \\ 
 \hline 
 $39$ & $U\oplus A_5\oplus [-4]$ & $95$ & $U\oplus  A_1^{\oplus 2}\oplus A_6$ & $151$ & $U(2)\oplus (A_1^{\oplus 3}\oplus_2 [-18])$ \\ 
 \hline 
 $40$ & $U\oplus (D_5\oplus_4 [-108])$ & $96$ & $U\oplus  D_7\oplus [-8]$ &  $152$ & $U(2)\oplus A_1^{\oplus 3}\oplus [-6]$ \\ 
 \hline 
 $41$ & $U\oplus (A_5\oplus_2 [-18])$ & $97$ & $U\oplus  A_1^{\oplus 3}\oplus D_5$ & $153$ & $U(2)\oplus A_1^{\oplus 3}\oplus [-8]$  \\ 
 \hline 
 $42$ & $U\oplus (D_5\oplus_2 [-28])$ & $98$ & $U\oplus  (E_7\oplus_2 [-66])$ & $154$ & $U(8)\oplus D_4$ \\ 
 \hline 
 $43$ & $U\oplus (A_5\oplus_3 [-42])$ & $99$ & $U\oplus  (A_7\oplus_4 [-72])$ & $155$ & $U(2)\oplus (A_3\oplus_2 [-68])$  \\ 
 \hline 
  $44$ & $U\oplus (A_5\oplus_6 [-186])$ & $100$ & $U\oplus  A_2\oplus (D_5\oplus_2 [-12])$ & $156$ & $([2]\oplus_2 [-10]) \oplus D_4(2)$ \\ 
 \hline 
 $45$ & $U\oplus A_1\oplus (A_4\oplus_5 [-80])$ & $101$ & $U\oplus  A_2\oplus D_5\oplus [-4]$ & $157$ & $U(3)\oplus A_1\oplus A_2\oplus [-6]$  \\ 
 \hline 
 $46$ & $U\oplus (D_5\oplus_2 [-36])$ & $102$ & $U\oplus  A_1^{\oplus 2}\oplus A_2\oplus D_4$ & $158$ & $U(3)\oplus A_2\oplus (A_1\oplus_2 [-30])$  \\ 
 \hline 
 $47$ & $U\oplus A_5\oplus [-6]$ & $103$ & $U\oplus  (D_7\oplus_2 [-52])$ & $159$ & $U(2)\oplus A_1^{\oplus 3}\oplus [-16]$ \\ 
 \hline 
 $48$ & $U\oplus (A_5\oplus_2 [-26])$ & $104$ & $U\oplus  A_3\oplus D_4\oplus [-4]$ & $160$ & $U(2)\oplus D_4\oplus [-4]$  \\ 
 \hline 
 $49$ & $U\oplus D_5\oplus [-10]$ & $105$ & $U\oplus  D_4\oplus (A_3\oplus_2 [-20])$ & $161$ & $U(2)\oplus D_4\oplus [-8]$ \\ 
 \hline 
 $50$ & $U\oplus A_1\oplus A_4 \oplus [-4]$ & $106$ & $U\oplus  A_1^{\oplus 5}\oplus A_3$ & $162$ & $U(2)\oplus D_4\oplus [-16]$  \\ 
 \hline 
 $51$ & $U\oplus D_4\oplus (A_1\oplus_2 [-22])$ & $107$ & $U\oplus  E_8\oplus [-8]$ & $163$ & $U(2)\oplus A_1^{\oplus 4}\oplus [-4]$ \\ 
 \hline 
 $52$ & $U\oplus A_1\oplus(A_4\oplus_5 [-120])$ & $108$ & $U\oplus  A_9$ & $164$ & $U(2)\oplus A_1^{\oplus 4}\oplus [-6]$  \\ 
 \hline 
 $53$ & $U\oplus A_1^{\oplus 2}\oplus(A_3\oplus_2 [-12])$ & $109$ & $U\oplus  A_3\oplus E_6$ & $165$ & $U(3)\oplus A_2^{\oplus 2}\oplus [-6]$ \\ 
 \hline 
 $54$ & $U\oplus A_1^{\oplus 4}\oplus A_2$ & $110$ & $U\oplus  E_8\oplus [-14]$ & $166$ & $U(2)\oplus A_1^{\oplus 5}\oplus [-4]$ \\ 
 \hline 
 $55$ & $U\oplus A_1\oplus (A_4\oplus_5 [-130])$ & $111$ & $U\oplus  E_8\oplus [-16]$ & & \\ 
 \hline 
 $56$ & $U\oplus D_4\oplus A_1 \oplus [-8]$ & $112$ & $U\oplus  A_1^{\oplus 2}\oplus D_7$ & & \\ 
 \hline

\end{tabular}
}
\end{center}
\medskip
{\scriptsize
In this table we use $L_1\oplus_k L_2$ to denote a primitive extension of $L_1\oplus L_2$ with $[L_1\oplus_k L_2: L_1\oplus L_2]=k$ and $L_1\oplus_k L_2$ being even. Such $L_1\oplus_k L_2$ here is unique up to isometries. $[n]$ is the rank $1$ lattice of determinant $n$.
}
\end{table}

\end{document}